\documentclass[a4paper,12pt]{amsart}
\usepackage{vmargin}
\usepackage[pdftex]{graphicx}
\usepackage{amsfonts}
\usepackage{amssymb}
\usepackage{mathrsfs}
\usepackage{stmaryrd}
\usepackage{amsmath}
\usepackage{amsthm}
\usepackage{enumerate}
\usepackage{nomencl}
\usepackage[bookmarks=true]{hyperref}   
\usepackage{esint}
\usepackage{dsfont}
\usepackage{upgreek}
\usepackage{color}

\makeatletter
\renewcommand\section{\@startsection{section}{1}{\z@}%
                       {-3\p@ \@plus -4\p@ \@minus -4\p@}%
                       {3\p@ \@plus 4\p@ \@minus 4\p@}%
                      {\normalfont\normalsize\centering\scshape}}
\makeatother

\hypersetup{
    colorlinks,%
}

\author{Lashi Bandara}
\author{Alan McIntosh}
\title[The Kato square root problem on GBG vector bundles]{The Kato square root problem on vector bundles with generalised bounded geometry}
\date{\today}
\address{Lashi Bandara, Centre for Mathematics and its Applications, 
Australian National University, Canberra, ACT, 0200, Australia}
\urladdr{\href{http://maths.anu.edu.au/~bandara}{http://maths.anu.edu.au/~bandara}}
\email{\href{mailto:lashi.bandara@anu.edu.au}{lashi.bandara@anu.edu.au}}
\address{Alan McIntosh, Centre for Mathematics and its Applications,
Australian National University, Canberra, ACT, 0200, Australia }
\urladdr{\href{http://maths.anu.edu.au/~alan}{http://maths.anu.edu.au/~alan}}
\email{\href{mailto:alan.mcintosh@anu.edu.au}{alan.mcintosh@anu.edu.au}}

\newtheorem{theorem}{Theorem}[section]
\newtheorem{corollary}[theorem]{Corollary}
\newtheorem{lemma}[theorem]{Lemma}
\newtheorem{proposition}[theorem]{Proposition}

\newtheorem{definition}[theorem]{Definition}
\newtheorem{remark}[theorem]{Remark}

\newcommand{\mdot}{\cdotp}

\newcommand{\cbrac}[1]{\left(#1\right)}
\newcommand{\bbrac}[1]{\left[#1\right]}
\newcommand{\dbrac}[1]{\left\{#1\right\}}
\newcommand{\modulus}[1]{\left|#1\right|}
\newcommand{\set}[1]{\dbrac{#1}}

\newcommand{\dom}{ {\mathcal{D}}}
\newcommand{\ran}{ {\mathcal{R}}}
\newcommand{\nul}{ {\mathcal{N}}}
\DeclareMathOperator{\card}{card}

\newcommand{\e}{\mathrm{e}}

\newcommand{\R}{\mathbb{R}}
\newcommand{\C}{\mathbb{C}}
 
\newcommand{\In}{\mathbb{Z}}

\newcommand{\Na}{\ensuremath{\mathbb{N}}}

\newcommand{\script}[1]{\mathscr{#1}}

\DeclareMathOperator{\re}{Re}			

\renewcommand{\emptyset}{\varnothing}
\newcommand{\union}{\cup}

\newcommand{\intersect}{\cap}

\newcommand{\rest}[1]{{{\lvert_{}}_{}}_{#1}}
\newcommand{\close}[1]{\overline{#1}}		

\newcommand{\ind}[1]{\raisebox{\depth}{\(\chi\)}_{#1}}	

\renewcommand{\epsilon}{\varepsilon}

\renewcommand{\phi}{\varphi}

\newcommand{\graph}{\script{G}}		

\newcommand{\tensor}{\otimes}

\newcommand{\comm}[1]{\bbrac{#1}}		

\newcommand{\norm}[1]{\left\| #1 \right\|}			
\newcommand{\spt}[1]{{\rm spt} {\text{ }}#1}	

\DeclareMathOperator{\tr}{tr}			
\DeclareMathOperator{\diam}{diam}		

\DeclareMathOperator{\len}{\ell}			
\DeclareMathOperator{\rad}{rad}			

\DeclareMathOperator{\divv}{div}		

\newcommand{\Christ}[2]{\Gamma^{#2}_{#1}}	

\newcommand{\Ric}{{\rm Ric}}			

\DeclareMathOperator{\inj}{inj} 		

\newcommand{\Tensors}[1][{}]{{\mathcal{T}}^{(#1)}}	
\newcommand{\Sect}{\mathbf{\Gamma}}		
\DeclareMathOperator{\proj}{\mathbf{P}}	
\newcommand{\tanb}{{\rm T}}		
\newcommand{\cotanb}{{\rm T}^\ast}	

\DeclareFontFamily{OT1}{restrictfont}{}
\DeclareFontShape{OT1}{restrictfont}{m}{n}{<-> fmvr8x}{}

\newcommand{\adj}[1]{{#1}^\ast}			
\newcommand{\extd}{{\rm d}}			

\newcommand{\inprod}[1]{\left\langle #1 \right\rangle}	
\newcommand{\grad}{\nabla}			
\newcommand{\conn}[1][{}]{{\grad_{{#1}}}}		
\DeclareMathOperator{\Lipp}{Lip}		

\newcommand{\bddlf}{\mathcal{L}} 	
\newcommand{\spec}{\sigma}		
\newcommand{\rset}{\rho}				

\newcommand{\Lp}[2][{}]{{\rm L}^{#2}_{\rm #1}}		
\newcommand{\Ck}[2][{}]{{\rm C}^{#2}_{\rm #1}}		
\newcommand{\Sob}[2][{}]{{\rm W}^{#2}_{\rm #1}}		
\newcommand{\SobH}[2][{}]{\Sob[#1]{#2,2}}

\newcommand{\ddelta}{\updelta}

\DeclareMathOperator{\sgn}{sgn}

\newcommand{\Hil}{\script{H}}			
\newcommand{\Banach}{\script{B}}			
\newcommand{\Lap}{\Delta}			

\newcommand{\Q}[1][{}]{Q_{#1}}			
\newcommand{\DyQ}{\script{Q}}			
\newcommand{\parent}[1]{\widehat{#1}}		

\newcommand{\brad}{\rho} 			
\newcommand{\scale}{\mathrm{t_S}}			
\newcommand{\hscale}{\mathrm{t_H}} 
\newcommand{\jscale}{\mathrm{J}}

\newcommand{\sC}{\script{C}}

\newcommand{\sN}{\script{N}}

\newcommand{\cE}{\mathcal{E}}
\newcommand{\cV}{\mathcal{V}}
\newcommand{\cM}{\mathcal{M}} 
\newcommand{\cW}{\mathcal{W}}

\newcommand{\Hol}{{\rm Hol}}

\newcommand{\Spa}{\mathcal{X}}

\newcommand{\mg}{\mathrm{g}}
\newcommand{\mh}{\mathrm{h}}

\newcommand{\tf}{\mathrm{f}}

\newcommand{\Poincare}{Poincar\'e~}		

\newcommand{\Av}{\mathcal{A}}
\newcommand{\pri}{\upgamma}

\newcommand{\Sq}{\mathcal{S}}

\newcommand{\Mul}{{\mathrm{M}}}

\newcommand{\Div}{\mathrm{L}}

\newcommand{\Carl}{\mathcal{C}}

\newcommand{\CBox}{\mathrm{R}}

\newcommand{\RNum}[1]{\uppercase\expandafter{\romannumeral #1\relax}}

\begin{document}

\maketitle
\vspace*{-2em}
\begin{abstract}
We consider smooth, complete Riemannian
manifolds which are
exponentially locally doubling.
Under a uniform Ricci curvature bound
and a uniform lower bound on injectivity radius,
we prove a Kato square root estimate 
for certain coercive operators
over the bundle of finite rank tensors.
These results are obtained as a special case
of similar estimates on smooth vector bundles 
satisfying a criterion
which we call \emph{generalised bounded geometry}.
We prove this by establishing quadratic estimates 
for perturbations of Dirac type operators on such bundles
under an appropriate set of assumptions.
\end{abstract}
\tableofcontents
\vspace*{-2em}

\parindent0cm
\setlength{\parskip}{\baselineskip}

\section{Introduction}

In this paper, we consider
the \emph{Kato square root problem}
for uniformly elliptic operators on smooth 
vector bundles $\cV$
over smooth, complete Riemannian manifolds $\cM$
which are at most exponentially locally doubling.
Let $\conn$ be a connection
on the bundle and $\mh$ its metric.
Define the uniformly elliptic operator 
$\Div_{A}u 
	= -a \divv (A_{11} \conn u) - a \divv(A_{10}u) 
	+ a A_{01} \conn u + a A_{00} u$
where $\divv = -\adj{\conn}$
and where the $a, A_{ij}$ are $\Lp{\infty}$
coefficients.
Under an appropriate bounded geometry 
assumption, we show that
$\dom(\sqrt{\Div_{A}}) = \dom(\conn) = \SobH{1}(\cV)$
and that
$\|\sqrt{\Div_{A}}u\| \simeq \norm{u}_{\SobH{1}}$
for all $u \in \SobH{1}(\cV)$.

The case of trivial, flat bundles
was considered 
by Morris 
in \cite{Morris3} (and his thesis \cite{Morris}).
In particular, he obtains solutions
to the Kato problem on Euclidean submanifolds
under  \emph{extrinsic} curvature bounds.
The novelty of our work is that
we dispense with the requirement of
an embedding and prove much more general
results under \emph{intrinsic} assumptions. 
We take the perspective of preserving the thrust
of the harmonic analytic argument
from the Euclidean context
and as a consequence, we are forced to
perform a  detailed and intricate  analysis of the geometry.

Our work utilises
the foundation laid
by Axelsson, Keith and
McIntosh in \cite{AKMC}, 
and the perspective
developed in \cite{AKM2}
by the same authors.
The ideas in both 
these papers
have their roots in
the solution of 
the Kato conjecture
by Auscher, Hofmann, Lacey,
McIntosh and Tchamitchian
in \cite{AHLMcT}.
See also the surveys
of Hofmann \cite{HofmannEx}, 
McIntosh \cite{Mc90},
the book by Auscher and
Tchamitchian \cite{AT},
and the recent survey
\cite{HMc} by Hofmann and McIntosh.

The idea of the authors of 
\cite{AKMC}
is to consider closed,
densely defined, nilpotent, 
operators
$\Gamma$, along with
perturbations $B_1,B_2$
and then to establish
quadratic estimates of
the form 
$$
\int_{0}^\infty \norm{t \Pi_B(I + t^2 \Pi_B)^{-1}u}^2 \frac{dt}{t}
	\simeq \norm{u}^2$$
where $\Pi_B = \Gamma + B_1 \adj{\Gamma} B_2.$
In \cite{AKM2}, 
the authors illustrate
that for \emph{inhomogeneous}
operators, it is enough 
to establish a certain
\emph{local} quadratic estimate, 
for which we need bounds on the integral 
from $0$ to $1$,
since the integral from $1$ to $\infty$
is straightforward in this case.
The proof of this quadratic
estimate proceeds by
reduction to a Carleson 
measure estimate.

The techniques
developed in \cite{AHLMcT}, 
\cite{AKMC}  and \cite{Morris3} 
rely upon being able to 
take averages of functions over subsets, and
defining constant vectors
in key aspects of the proof.
This is a primary obstruction
to generalising these techniques
to non-trivial bundles.
To circumvent this obstacle,
we formulate a
condition which we
call \emph{generalised bounded geometry}.
This condition captures
a uniform locally Euclidean structure
in the bundle.
The existence of a 
dyadic decomposition
(below a fixed scale)
provides a
decomposition of the manifold
in a way that
allows us to work on a fixed set of
coordinates in the bundle.
We can picture this 
decomposition of the bundle
as a sort of abstract polygon -
Euclidean regions
separated by a boundary of
null measure.
Under the condition of generalised bounded geometry, 
and using this decomposition, 
we are then able to adapt the arguments 
of \cite{Morris3}
and \cite{AKMC} 
in order to obtain a Kato square root estimate.

 The intuition behind
generalised bounded geometry
is the existence of
\emph{harmonic coordinates}
under Ricci
bounds, along with a uniform lower bound on injectivity radius.
 We use this 
to show that the bundle of
$(p,q)$ tensors satisfies
generalised bounded geometry. 
 As a consequence, 
we obtain a Kato square root
estimate for tensors
under a coercivity condition which
is automatically satisfied
for scalar-valued functions.
We highlight the scalar-valued version as a central
theorem of this paper.

\begin{theorem}[Kato square root problem for functions]
\label{Thm:Int:KatoFn}
Let $(\cM,\mg)$ be a smooth, complete Riemannian manifold with 
$\modulus{\Ric} \leq C$ and  $\inj(M) \geq \kappa > 0$.
Suppose that
the following ellipticity
condition holds:
there exist $\kappa_1, \kappa_2 > 0$ such that 
\begin{align*}
\re\inprod{a u, u}  &\geq \kappa_1 \norm{u}^2 
\qquad\text{and}\\
\re(\inprod{A_{11} \nabla v, \nabla v} +\inprod{A_{10}  v, \nabla v}+\inprod{A_{01} \nabla v,  v}+\inprod{A_{00} v,  v} )&\geq \kappa_2 \norm{v}_{\SobH{1}}^2
\end{align*} 
for all $u \in \Lp{2}(\cM)$
and $v \in \SobH{1}(\cM)$.
Then, 
$\dom(\sqrt{\Div_A}) = \dom(\conn) = \SobH{1}(\cM)$
and $\|{\sqrt{\Div_A} u}\|\simeq \norm{\conn u} + \norm{u} = \norm{u}_{\SobH{1}}$
for all $u \in \SobH{1}(\cM)$.
\end{theorem}

We also prove
Lipschitz estimates
for small perturbations of our operators,
similar to those in \S6 of \cite{AKMC},
which are a direct consequence of
considering operators with \emph{complex} 
bounded measurable coefficients. 

In \cite{B}, the first author
proves quadratic estimates
for operators $\Pi_B$
on doubling measure metric 
spaces under an appropriate set of 
assumptions.
We conclude this paper
by demonstrating how to extend
the quadratic estimates which
we obtain on a manifold,
to the more general setting
of a complete metric space
equipped with a Borel-regular
measure that is exponentially locally
doubling.

\section*{Acknowledgements}

The authors appreciate the support of the Centre for Mathematics and its Applications 
at the Australian National University, Canberra, where this project was undertaken. 
The first author was supported through an Australian Postgraduate Award
and through the Mathematical Sciences Institute, Australian National University, Canberra.
The second author gratefully acknowledges support from the 
Australian Government through the Australian Research Council. 

The authors thank Andrew Morris for helpful conversations and suggestions.
\section{Preliminaries}\label{prelim}

\subsection{Notation}
Throughout this paper, we use the
Einstein summation convention. That is,
whenever there is a  repeated lowered index 
and a raised index (or conversely), 
we assume summation over that
index. 
By $\Na$ we denote natural 
numbers not including $0$
and we let $\In_{+} = \Na \union \set{0}$.
 We denote the integers by $\In$. 
We take the liberty
to sometimes write $a \lesssim b$ 
for two real quantities $a$ and $b$.
By this, we mean that there is
a constant $C > 0$ such that
$a \leq C b$. By $a \simeq b$ 
we mean that $a \lesssim b$
and $b \lesssim a$.
For a function (and indeed,
a section) $f$,
we denote its support by $\spt f$.
We denote an open ball centred at $x$ 
of radius $r$ by $B(x,r)$.
The radius of a ball $B$ (open or closed) 
is denoted by $\rad(B)$.
For $\Omega \subset \cM$,
we denote the diameter of $\Omega$ 
by $\diam \Omega = \sup\set{ d(x,y):\, x,y\in\Omega}$.
 
\subsection{Manifolds and vector bundles}
 
In this section, we introduce terminology
that allows us to describe the class
of manifolds in which we obtain our results. 
Furthermore, we introduce the function spaces
that we shall use. We also prove a key 
result that allows us to
construct Sobolev spaces on vector bundles. 

Let $\cM$ be a smooth, 
complete Riemannian manifold with 
metric $\mg$ and Levi-Civita connection $\conn$.
Let $d\mu$ denote the volume measure induced
by $\mg$. We follow
the notation of \cite{Morris3} in the following
definition.

\begin{definition}[Exponential volume growth]
\index{Exponential volume growth}
We say that
$\cM$ has \emph{exponentially locally doubling volume growth}
if there exists $c \geq 1,\ \kappa, \lambda \geq 0$
such that
\begin{equation*}
0 < \mu(B(x,tr)) \leq 
c t^\kappa \e^{\lambda t r} \mu(B(x,r)) < \infty
\tag{$\text{E}_{\text{loc}}$}
\label{Def:Pre:Eloc}
\end{equation*}
for all $t \geq 1,\ r > 0$ and $x \in \cM$.
\end{definition}

Let $\cV$ be a smooth 
vector bundle of rank $N$
with metric $\mh$ and connection $\conn$.
Let $\uppi_{\cV}: \cV \to \cM$ denote 
the canonical projection.
Since we are interested in considering
``sections'' which may have low regularity,
we deviate from the usual definition of 
$\Sect(\cV)$ by dropping the requirement
that they be differentiable. More precisely,
we write $\Sect(\cV)$ to be the space
of measurable 
functions $\omega:\cM \to \cV$ such that
$\omega(x) \in \uppi_{\cV}^{-1}(x)$. We impose
no regularity assumptions (other than measurability)
and we call 
these \emph{sections} of $\cV$.
We write $\Lp[loc]{1}(\cV) = \Lp[loc]{1}(\cM,\cV) $ to denote
the sections $\gamma \in \Sect(\cV)$
such that
$$ \int_{K} \modulus{\gamma(x)}_{\mh(x)}\ d\mu(x) < \infty$$
for each compact $K \subset \cM$.
Similarly, we define
$\Lp{2}(\cV)$. The spaces $\Ck{k}(\cV)$
are then the $k$-differentiable sections
and $\Ck[c]{k}(\cV)$ are $k$-differentiable 
sections with compact support.
The following proposition is necessary 
to define Sobolev spaces
on $\cV$.
 
\begin{proposition}
\label{Prop:Pre:ConnClosed}
The connection $\conn: \Ck{\infty} \intersect \Lp{2}(\cV) 
\to \Lp{2}(\cotanb \cM \tensor \cV)$ is 
a densely defined, closable operator.
\end{proposition}
\begin{proof}
That $\Ck[c]{\infty}(\cV)$
is dense in $\Lp{2}(\cV)$
is shown by a mollification argument after covering the manifold by countably many compact
sets corresponding to coordinate charts.
Since $\Ck[c]{\infty}(\cV) \subset \Ck{\infty} \intersect \Lp{2}(\cV)$,
we have that $\conn$ is a densely defined operator.

We show that $\conn$ is a closable operator by reduction
to the well known fact that $\conn = \extd$ is closable
on scalar-valued functions. 
Fix $u_n \in \Ck{\infty} \intersect \Lp{2}(\cV)$
such that $u_n \to 0$ and $\conn u_n \to v$.
For each $x \in \cM$, we can find an 
open set $K_x$ such that $\close{K_x}$
is compact and with a local trivialisation 
$\psi_x : \close{K_x} \times \C^N \to \uppi_{\cV}^{-1}(\close{K_x})$.
Let $\set{K_x}$ be a collection of such sets
and let $\set{K_j}$ denote a countable
subcover. Now, fix $K_j$ and
let $\set{e^i}$ denote an orthonormal 
frame (by a Gram-Schmidt procedure) 
in $\close{K_j}$. 
Write $u_n = (u_n)_i e^i$
and note that 
$$
\conn u_n = \conn (u_n)_i \tensor e^i + (u_n)_i \conn e^i$$
and therefore,
$$
\sum_{i} \modulus{\conn (u_n)_i - v_i}^2 
	= \modulus{\conn (u_n)_i \tensor e^i - v}^2
	\leq \modulus{\conn u_n - v}^2 + \sup_{l} \modulus{\conn e^l}^2 \modulus{u_n}^2.$$
Since by assumption $\close{K_i}$ is compact
and the basis $e^i$ is smooth,
we have that $\sup_{l} \modulus{\conn e^l}^2 \leq C_j$
for some $C_j$ dependent on $K_j$. Thus,
$$\norm{ \conn(u_n)_i - v_i}^2_{\Lp{2}(K_j)} \to 0.$$
Thus, we have that 
$(u_n)_i \to 0$ and $\conn (u_n)_i \to v_i$
in $\Lp{2}(K_j)$, 
and by the closability of $\conn$
on functions, we have that $v_i = 0$ almost everywhere
in $K_j$. Thus, $v = 0$ almost everywhere in
$K_j$ and consequently $v = 0$ almost everywhere in $\cM$.
This proves that $\conn$ is a closable operator.
\end{proof} 

As a consequence of this proposition, 
we can define the Sobolev space
$\SobH{1}(\cV)$
as the completion of functions
$\Ck{\infty} \intersect \Lp{2}(\cV)$
in the graph norm of $\conn$.  
The closure of $\conn$ is then denoted by the same symbol, namely
$\conn: \SobH{1}(\cV) \subset \Lp{2}(\cV) \to \Lp{2}(\cotanb \cM \tensor \cV)$.

The higher order Sobolev 
spaces $\SobH{k}(\cV)$
are defined as subsets of $\SobH{1}(\cV)$
in the usual manner.
To keep some accord with tradition,
when we consider the situation
$\cV = \C$,
we write $\Lp{2}(\cM)$ in place of $\Lp{2}(\cV)$
and similarly for the function spaces $\Ck{k},\ \Ck[c]{k}$
and $\SobH{k}$.

We also highlight that we have the following
important density result.
\begin{proposition}
$\Ck[c]{\infty}(\cV)$ is dense in $\SobH{1}(\cV)$.
\end{proposition}
\begin{proof}
Since we assume that $\cM$ is complete, 
the Hopf-Rinow theorem guarantees that
closed balls are compact. Fix some centre $x \in \cM$
and consider closed balls $\close{B(x,r)}$.
For such balls, we can find $\eta_r \in \Ck[c]{\infty}(\cV)$
such that $\spt \eta_r$ is compact,
$0 \leq \eta \leq 1$, $\eta_r \equiv 1$ on $B(x,r)$ and
$\sup \modulus{\conn \eta_r} \to 0$ as $r \to \infty$.
Now, for any $u \in \SobH{1}(\cV)$,
$\eta_r u$ is compactly supported. Also,
$$ 
\conn(\eta_r u) = \conn \eta_r \tensor u + \eta_r \conn u$$
and 
$$
 \norm{\conn(\eta_r u) - \conn u} 
	\leq  \norm{\conn \eta_r \tensor u } 
		+  \norm{\eta_r \conn u - \conn u}
	\leq \sup_{x \in \cM} \modulus{\conn \eta_r} \norm{u} 
		+  \norm{\eta_r \conn u - \conn u} \to 0$$
as $r \to \infty$.
Thus, it suffices to consider $u \in \SobH{1}(\cV)$
with $\spt u$ compact. But for each such $u$,
an easy mollification argument in each 
local trivialisation  will yield a sequence $u_n \in \Ck[c]{\infty}(\cV)$
such that $\norm{u_n - u}_{\SobH{1}} \to 0$. 
\end{proof}

\begin{remark}
We remark that this proof does not generalise to higher order
Sobolev spaces. In fact, even for $k = 2$,
the best known result 
for the density of $\Ck[c]{\infty}(\cM)$
in $\SobH{2}(\cM)$ 
is under uniform lower bounds on injectivity radius
along with uniform lower bounds on Ricci curvature.
See \S3 of \cite{Hebey} and in particular, 
Proposition 3.3 in \cite{Hebey}.
\end{remark} 

For $f \in \Lp[loc]{1}(\cM)$ (i.e. a function), we define
the \emph{average} of $f$ on some
measurable subset $A \subset \cM$
with $0 < \mu(A) < \infty$ by 
$ f_A = \fint_{A} f\ d\mu = \frac{1}{\mu(A)} \int_{A} f\ d\mu.$

In what follows, we  assume that 
$\cM$ satisfies (\ref{Def:Pre:Eloc})
unless stated otherwise.

\subsection{Generalised Bounded Geometry}
\label{Sect:Prelim:GBG}
 
The harmonic analytic  techniques  we employ
in the main proof, along with some of the
assumptions we make on the operators under 
 consideration, require us to capture a uniform
locally Euclidean structure in the underlying
vector bundle. The concept which
we describe here is
motivated by the fact that 
injectivity radius bounds on a manifold,
coupled with appropriate curvature
bounds, give \emph{bounded geometry}
on the $(p,q)$ tensor bundle.
It  provides us with the framework for applying a dyadic decomposition in order
to construct a system of global
coordinates (no longer smooth), thus allowing us to define
key quantities and to construct proofs.

Recalling that $\cV$ is equipped with a metric $\mh$, we make the
 following definition.

\begin{definition}[Generalised Bounded Geometry]\label{GBG}
\label{Def:Pre:GBG}
\index{Generalised Bounded Geometry}
Suppose there exists $\brad > 0$, $C \geq 1$,
such that
for each $x \in \cM$, there exists a local trivialisation 
$\psi_x: B(x,\brad) \times \C^N \to \uppi_{\cV}^{-1}(B(x,\brad))$ 
satisfying
$$
C^{-1} I \leq \mh \leq C I$$
in the basis $\set{e^i = \psi_x (y,\tilde{e}^i)}$,
where $\set{\tilde{e}^i}$ is the standard
orthonormal basis for $\C^N$.
Then, we say that $\cV$ has \emph{generalised bounded geometry}
or \emph{GBG}. 
We call $\brad$ the \emph{GBG radius},
and local trivialisations $\psi_x$ the \emph{GBG charts}. 
\end{definition}

Since we can always take $\brad$ to be
as small as we like, we assume
that $\brad \leq 5$.

For the convenience of the reader
we quote Proposition 4.2 from \cite{Morris3}.

\begin{theorem}[Existence of a truncated dyadic structure]
\label{Thm:Dya:Christ}
There exists a countable collection of open subsets
$\set{\Q[\alpha]^k \subset \cM: \alpha \in I_k,\ k \in \In_{+}}$,
$z_\alpha^k \in \Q[\alpha]^k$ (called the \emph{centre} of $\Q[\alpha]^k$), 
index sets $I_k$ (possibly finite), 
and constants $\delta \in (0,1)$, 
$a_0 > 0$, $\eta > 0$ and $C_1, C_2 < \infty$ satisfying:
\begin{enumerate}[(i)]
\item for all $k \in \In$,
	$\mu(\cM \setminus \union_{\alpha} \Q[\alpha]^k) = 0$,
\item if $l \geq k$, then either $\Q[\beta]^l \subset \Q[\alpha]^k$
	or $\Q[\beta]^l \intersect \Q[\alpha]^k = \emptyset$,
\item for each $(k,\alpha)$ and each $l < k$
	there exists a unique $\beta$ such that
	$\Q[\alpha]^k \subset \Q[\beta]^l$,
\item $\diam \Q[\alpha]^k < C_1 \delta^k$,
\item $ B(z_\alpha^k, a_0 \delta^k) \subset \Q[\alpha]^k$,  
\item for all $k, \alpha$ and for all $t > 0$, 
	$\mu\set{x \in \Q[\alpha]^k: d(x, \cM\setminus\Q[\alpha]^k) \leq t \delta^k} \leq C_2 t^\eta \mu(\Q[\alpha]^k).$
\end{enumerate} 
\end{theorem}

\begin{remark}
This theorem was first proved by Christ  in \cite{Christ} for $k\in\In$,
for doubling measure metric spaces.
\end{remark}

Throughout this paper we fix $\jscale  \in \Na$ such that
$C_1 \delta^\jscale \leq \frac{\brad}{5}$.  For $j \geq \jscale$, we write $\DyQ^j$
to denote the collection of all cubes $\Q[\alpha]^j$,
and set $\DyQ = \union_{j \geq \jscale} \DyQ^j$. 

We call $\scale = \delta^{\jscale}$ the \emph{scale}. Furthermore, 
when $t\leq \scale$, set  $\DyQ_t = \DyQ^j$ whenever
$\delta^{j+1} < t \leq \delta^{j}$.

The existence of a
truncated dyadic structure
allows us to formulate the following
 system of coordinates on $\cV$. 

\begin{definition}[GBG coordinates of $\cV$]
\index{Generalised Bounded Geometry!Coordinates}
\index{Generalised Bounded Geometry!Dyadic coordinates}
\index{Scale}
Let $x_{\Q}$ denote the \emph{centre}
of each $\Q \in \DyQ^\jscale$.
Then, we call the system of coordinates
$$\sC = \set{\psi:B(x_{\Q},\brad) \times \C^N \to \uppi_{\cV}^{-1}(B(x_{\Q},\brad))\ \text{s.t.}\ \Q \in \DyQ^\jscale}$$ 
the \emph{GBG coordinates}.
We call the set 
$$\sC_{\jscale} = \set{\psi_{\Q} = \psi\rest{\Q}: \Q \times \C^N \to \uppi_{\cV}^{-1}(\Q)\ \text{s.t.}\ \Q \in \DyQ^\jscale}$$
the \emph{dyadic GBG coordinates}.
For any cube $Q \in \DyQ^j$, 
there is a unique
cube $\parent{\Q} \in \DyQ^{\jscale}$
satisfying $\Q \subset \parent{\Q}$
and we call this cube the \emph{GBG cube
of $\Q$}.
The \emph{GBG coordinate system}
of $\Q$ is then $\psi:B(x_{\parent{\Q}},\brad) \times \C^N \to \uppi_{\cV}^{-1}(B(x_{\parent{\Q}},\brad)$.
\end{definition}

\begin{remark} 
Since $\union \DyQ^J$ 
is an almost-everywhere covering
of $\cM$,
the GBG coordinate system of $\cV$
defines an almost-everywhere smooth
global trivialisation of the vector bundle.
\end{remark} 

We emphasise that throughout
this paper,
for any cube $\Q \in \DyQ^j$, 
we denote its GBG 
cube by $\parent{\Q}$.

A first and important consequence of
the GBG coordinates is that
it allows us to define
the notion of a \emph{constant} section
(locally, in the eyes of our GBG coordinates and dyadic decomposition).
More precisely, for $x \in \Q$, 
given $w = w_i\ e^i(x) \in \cV_x$,
where the $\set{e^i}$ is the GBG coordinates 
of $\Q$,
we define $\omega(y) = w_i\ e^i(y)$ whenever 
$y \in B(x_{\parent{Q}},\rho)$ and $\omega(y) = 0$
otherwise.
This is an extension of $w \in \cV_x$ to
the entire GBG coordinate ball $B(x_{\parent{\Q}},\rho)$.
Thus, we call $\omega$ the GBG constant
section associated to $w$. 
Note that $\omega\in \Lp{\infty}(\cV)$.  We remark that
this notion is crucial in later parts of the paper.
Next, we can define a notion of \emph{cube integration}
in the following way.
Let $u \in \Lp[loc]{1}(\cV)$. 
Then, for any $\Q \in \DyQ$,
we can write
$$ \int_{\Q} u\ d\mu  = \cbrac{\int_{\Q} u_i\ d\mu} e^i$$
where $u = u_i e^i$ in the GBG coordinates of $\Q$. 
Note that $ \int_{\Q} u\ d\mu$ is a function from $Q$ to $\cV$.

Pursuing a similar vein of thought, we define 
a \emph{cube average}. 
Given a cube $\Q \in \DyQ$ 
and $u \in \Lp[loc]{1}(\cV)$,
define $u_{\Q}\in \Lp{\infty}(\cV)$ by
$$ 
u_{\Q}(y) =  \begin{cases}
	\tfrac{1}{\mu(\Q)} \int_{\Q}u\ d\mu &\text{when } y\in B(x_{\parent{\Q}},\brad),\\
	0 &\text{otherwise}.
	\end{cases}$$ 

\begin{remark}
We remark that
in the average of a function
over a general measurable set
of positive measure,
we do not perform a cutoff
as we do here. However,
whenever we write $u_{\Q}$
with $\Q$ being a cube (even for a function),
we shall always assume this definition.
\end{remark}

Lastly, we define the \emph{dyadic averaging operator}.
For each $t > 0$, define 
the operator $\Av_t: \Lp[loc]{1}(\cV) \to \Lp[loc]{1}(\cV)$
by 
$$
\Av_t u (x) = \frac{1}{\mu(\Q)} \int_{\Q} u(y)\ d\mu(y)$$
whenever $x \in \Q \in \DyQ_t$. 
We remark that $\Av_t: \Lp{2}(\cV) \to \Lp{2}(\cV)$ is 
a bounded operator with norm bounded uniformly for all $t \leq \scale$ by a  
bound depending on the constant $C$ arising in the GBG criterion. 

\subsection{Functional calculi of sectorial operators}

Of fundamental importance to the 
setup and proof that we present in this 
paper, 
is the functional calculus of certain operators.
In this section, we introduce the key type of
operators that we shall concern ourselves with,
and, for the convenience of the reader, 
recall some facts about
functional calculi of these operators.  
A fuller treatment of this material 
can be found in \cite{AMX}. A local version of this theory 
can be found in \cite{Morris2}.

Let $\Banach_1, \Banach_2$ be  Banach spaces.
We say that a linear map $T:\dom(T) \subset \Banach_1 \to \Banach_2$
is an \emph{operator} with domain $\dom(T)$.
The range of $T$ is denoted by $\ran(T)$
and the null space by $\nul(T)$.
We say that such an operator
is \emph{closed}
if the graph $\graph(T) = \set{(u, Tu): u \in \dom(T)}$
is closed in the product topology of 
$\Banach_1 \times \Banach_2$.
The operator $T$ is \emph{bounded} if
$\norm{Tu}_{\Banach_2} \leq \norm{u}_{\Banach_1}$
for all $u \in \dom(T) = \Banach_1$.
If $\Banach_1=\Banach_2$ then the \emph{resolvent} set $\rset(T)$ 
consists of all $\zeta \in \C$ such that
$\zeta I - T$ is one-one, onto and has a bounded inverse.
The spectrum is then $\spec(T) = \C \setminus \rset(T)$.

For $0 \leq \omega < \frac{\pi}{2}$, define the
bisector
$S_{\omega} = \set{\zeta \in \C: \modulus{\arg \zeta} \leq \omega\ \text{or}\ \modulus{\arg (-\zeta)} \leq \omega\ \text{or}\ \zeta = 0}$.
The open bisector is then defined as
$S_{\omega}^o = \set{\zeta \in \C: \modulus{\arg \zeta} < \omega\  \text{or}\ \modulus{\arg (-\zeta)} < \omega,\ \zeta \neq 0}$.
An operator $T$ in a Banach space $\Banach$  is  \emph{$\omega$-bisectorial}
if it is closed, $\spec(T) \subset S_{\omega}$,
and  the following \emph{resolvent bounds} hold: 
for each $\omega < \mu < \frac{\pi}{2}$, there exists
$C_{\mu}$ such that for all $\zeta \in \C\setminus S_\mu$, 
we have $\modulus{\zeta} \norm{(\zeta I - T)^{-1}} \leq C_{\mu}$.
Depending on context, we refer to such operators
simply as bisectorial and $\omega$ is then the
\emph{angle of bisectoriality}.

Now fix $0 \leq \omega < \mu < \frac{\pi}{2}$ 
and assume that $T$ is an $\omega$-bisectorial 
operator  in a Hilbert space $\Hil$. 
We let $\Psi(S_{\mu}^o)$ denote the space of all
holomorphic functions $\psi: S_{\mu}^o \to \C$ 
for which
$$\modulus{\psi(\zeta)} \lesssim \frac{\modulus{\zeta}^{\alpha}}{1 + \modulus{\zeta}^{2\alpha}}$$
for some $\alpha > 0$.
For such functions, we define a functional calculus 
similar to the Riesz-Dunford functional calculus
by
$$ 
\psi(T) = \frac{1}{2\pi\imath} \oint_{\gamma} \psi(\zeta)(\zeta I - T)^{-1}\ d\zeta$$
where $\gamma$ is a contour in $S_{\mu}^o$ enveloping $S_{\omega}$ 
parametrised anti-clockwise, and
the integral is defined via Riemann sums.
This integral converges absolutely
as a consequence of the decay of $\psi$, 
coupled with the resolvent bounds of $T$.
The operator $\psi(T)$ is bounded.
We say that $T$ has a \emph{bounded holomorphic
 functional calculus}
if there exists $C > 0$
such that $\norm{\psi(T)} \leq C \norm{\psi}_\infty$
for all $\psi \in \Psi(S_{\mu}^o)$. 

Now suppose that $\Hil$ is a Hilbert space, and  that $T:\dom(T)\subset \Hil \to \Hil$ is  $\omega$-bisectorial. In this case $\Hil = \nul(T)\oplus \overline{\ran(T)}$ where the direct sum is typically non-orthogonal.
Let  $\proj_{\nul(T)}$ denote the projection of $\Hil$
onto $\nul(T)$ which is $0$ on $\ran(T)$.
Let $\Hol^\infty(S_{\mu}^o)$ denote 
the space of all bounded functions $f: S_{\mu}^o \union \set{0} \to \C$
which are holomorphic on $S_{\mu}^o$.
For such a function $f$, there exists a uniformly bounded
sequence $(\psi_n)$  of functions in $\Psi(S_{\mu}^o)$ which
converges to $f\rest{S_{\mu}^o}$
in the compact-open topology over $S_{\mu}^o$.
If  $T$ has a bounded holomorphic
functional calculus, and $u\in\Hil$, then
the limit $\lim_{n\to\infty} \psi_n(T)u$ 
exists in $\Hil$, and we define
$$f(T)u = f(0)\proj_{\nul(T)}u + \lim_{n\to\infty} \psi_n(T)u\ .$$
This defines a bounded operator 
independent of
the sequence $(\psi_n)$, and 
we have the bound $\norm{f(T)} \leq C \norm{f}_\infty$ for some finite $C$. 

In particular the functions $\chi^{+}, \chi^{-}$ and $\sgn=\chi^{+}-\chi^{-}$ 
belong to $\Hol^\infty(S_{\mu}^o)$, 
where $\chi^{+}(\zeta) = 1$ when $\re(\zeta) > 0$
and $0$ otherwise, and 
$\chi^{-}(\zeta) = 1$ when $\re(\zeta) < 0$
and $0$ otherwise.
Therefore, if $T$ has a bounded 
functional calculus in $\Hil$, 
then $\chi^{\pm}(T)$ are bounded projections, 
and $\Hil = \nul(T) \oplus \ran(\chi^{+}(T))\oplus \ran(\chi^{-}(T))$. 
The bounded operator $\sgn(T)$ 
equals 0 on $\nul(T)$, equals 1 on 
$\ran(\chi^{+}(T))$, and equals $-1$ on $\ran(\chi^{-}(T))$. 
\section{Hypotheses}
\label{Sect:Ass}

Though our primary goal in this paper
is to provide,
under an appropriate set of
assumptions, a 
positive answer to the Kato square root problem
on vector bundles, we shall pursue a
slightly more general setup in the footsteps
of \cite{AKMC}.
The purpose of this section 
is to describe this setup as a
list of hypotheses, and in later sections,
demonstrate how to apply tools
from harmonic analysis to 
prove quadratic estimates under  
these hypotheses.
The Kato square root estimate
on vector bundles (and $(p,q)$ tensors) will
then be obtained by showing that 
these hypotheses are satisfied
under our geometric assumptions. 

Let $\Hil$ be a Hilbert space
and let $\inprod{\mdot,\mdot}$ 
denote its inner product.
Following \cite{AKMC}
and \cite{Morris3}, we make
the following operator theoretic assumptions.

\begin{enumerate}
\item[(H1)]
	The operator $\Gamma: \dom(\Gamma) \subset \Hil \to \Hil$
	is closed, densely defined and \emph{nilpotent}.

\item[(H2)]
	The operators $B_1, B_2 \in \bddlf(\Hil)$
	satisfy
	\begin{align*}
	&\re\inprod{B_1u,u} \geq \kappa_1 \norm{u}^2 &&\text{whenever }u \in \ran(\adj{\Gamma}), \\
	&\re\inprod{B_2u,u} \geq \kappa_2 \norm{u}^2 &&\text{whenever }u \in \ran(\Gamma),
	\end{align*}
	where $\kappa_1, \kappa_2 > 0$ are constants.

\item[(H3)]
	The operators $B_1, B_2$ satisfy
	$B_1B_2 (\ran(\Gamma)) \subset \nul(\Gamma)$
	and $B_2B_1(\ran(\adj{\Gamma})) \subset \nul(\adj{\Gamma})$.
\end{enumerate}

Furthermore, define
$\adj{\Gamma}_B = B_1 \adj{\Gamma} B_2$,
$\Pi_B = \Gamma + \adj{\Gamma}_B$
and $\Pi = \Gamma + \adj{\Gamma}$. 
The operator $\Pi$ is self-adjoint, and $\Pi_B$ is bisectorial,
and thus we define the following
associated bounded operators:
$$
R_t^B = (1 + it \Pi_B)^{-1},\ 
P_t^B = (1 + t^2 \Pi_B^2)^{-1},\ 
Q_t^B = t\Pi_B (1 + t^2 \Pi_B^2)^{-1},\ 
\Theta_t^B = t \adj{\Gamma}_B (1 + t^2 \Pi_B^2)^{-1}.$$
We write $R_t, P_t, Q_t,
\Theta_t$ on setting $B_1 = B_2 = 1$.
The full implications of these
assumptions are
listed in \S4 of \cite{AKMC}.

The following additional assumptions are
mild, particularly since
we wish to apply the theory to differential 
operators. They are essentially the same
as the assumptions made in 
\cite{AKMC} and \cite{Morris3}, but modified for vector bundles.
 We remark that in  (H5) below,
 $\bddlf(\cV)$ denotes
 the vector bundle 
of all bounded linear maps $T_x:\cV_x \to \cV_x$
for each $x \in \cM$
(where $\cV_x$ is the fibre over $x$).
The boundedness 
is with respect to the metric
$\mh_x$ on the fibre $\cV_x$. Note that  
the local trivialisations
for this bundle are canonically induced
by the local trivialisations of $\cV$, and in each local trivialisation
the $T_x$ can be represented
as usual by an $N \times N$
matrix.

\begin{enumerate}
\item[(H4)]	
	The Hilbert space $\Hil = \Lp{2}(\cV)$,
	 where $\cV$ is a smooth vector bundle with  smooth  metric
	$\mh$ over a smooth,  complete  Riemannian manifold $\cM$
	with  smooth metric $\mg$. 
	Furthermore, $\cV$ satisfies
	the GBG criterion and $\cM$ satisfies
	(\ref{Def:Pre:Eloc}). 

\item[(H5)]
	The operators $B_1, B_2$ are multiplication operators, i.e. there exist 
	$B_i \in \Lp{\infty}(\bddlf(\cV))$.
\item[(H6)]
	The operator $\Gamma$ is a first order differential 
	operator. That is, there exists  $C_{\Gamma} > 0$
	such that whenever $\eta \in \Ck[c]{\infty}(\cM)$,
	we have that 
	$\eta \dom(\Gamma) \subset \dom(\Gamma)$
	and 
	$\Mul_{\eta}u(x) = \comm{\Gamma,\eta(x)I}u(x)$
	is a multiplication operator satisfying
	$$ \modulus{\Mul_{\eta}u(x)} \leq C_{\Gamma} \modulus{\conn \eta}_{\cotanb M} \modulus{u(x)}$$
	for all $u \in \dom(\Gamma)$ and almost all $x \in \cM$.
\end{enumerate}

\begin{remark}
We note as in \cite{Morris3}
that (H6) implies the same hypothesis with
$\Gamma$ replaced by either $\adj{\Gamma}$ or
$\Pi$.
\end{remark}

It is in the following two hypotheses that we 
make a more substantial departure from
\cite{AKMC} and \cite{Morris3}. A significant
difference is that we have used the dyadic
structure, rather than balls,
in their formulation. This cannot
be avoided since we are forced to employ quantities
which are defined through GBG coordinates.

Recalling the definition
of a cube integral
in \S\ref{Sect:Prelim:GBG}, we formulate the following 
\emph{cancellation} hypothesis.  

\begin{enumerate}
\item[(H7)] There exists $c > 0$
	such that for all  $\Q \in \DyQ$,
	$$ \modulus{\int_{\Q} \Gamma u\ d\mu}
		\leq c \mu(\Q)^{\frac{1}{2}} \norm{u}
	\quad\text{and}\quad
	\modulus{\int_{\Q} \adj{\Gamma} v\ d\mu}
		\leq c \mu(\Q)^{\frac{1}{2}} \norm{v}$$	
	for all $u \in \dom(\Gamma),\ v\in \dom(\adj{\Gamma})$
	satisfying $\spt u,\ \spt v \subset \Q$.
\end{enumerate} 

Lastly, we make the following
abstract \Poincare and coercivity hypotheses 
on the bundle  (recalling that $\DyQ_t = \DyQ^j$ whenever
$\delta^{j+1} < t \leq \delta^{j}$).

\begin{enumerate}
\item[(H8)]
	There exist $C_P,C_C, c,  \tilde{c} > 0$
	and an operator 
	$\Xi: \dom(\Xi) \subset \Lp{2}(\cV) \to \Lp{2}(\sN)$,
	where $\sN$ is a normed bundle over $\cM$
	with norm $\modulus{\mdot}_{\sN}$
	and $\dom(\Pi) \intersect \ran(\Pi) \subset \dom(\Xi)$,
	satisfying: 
	\begin{enumerate}[-1]
	\item (Dyadic \Poincare)
		$$ \int_{B} \modulus{u - u_{\Q}}^2\ d\mu
			\leq C_P(1 + r^\kappa \e^{\lambda c rt})
			(rt)^2 
			\int_{\tilde{c}B} \cbrac{\modulus{\Xi u}_{\sN}^2 + \modulus{u}^2}\ d\mu$$
	for all balls $B = B(x_{\Q},rt)$ with $r \geq C_1/\delta$
	where $\Q \in \DyQ_t$ with $t \leq \scale$, and
	\item (Coercivity)
		$$ \norm{\Xi u}_{\Lp{2}(\sN)}^2 + \norm{u}_{\Lp{2}(\cV)}^2 
			\leq C_C \norm{\Pi u}_{\Lp{2}(\cV)}^2.$$
	\end{enumerate}
 for all $u \in \dom(\Pi)\intersect \ran(\Pi)$. 
\end{enumerate}

We justify calling this an abstract \Poincare
inequality for two reasons. First,
the inequality is for sections on a vector bundle,
not just for scalar-valued functions.
Second, in the usual \Poincare
inequality, the operator $\Xi$ is
simply $\conn$.
We allow ourselves other possibilities in choosing the  operator $\Xi$
here, because it can be useful
in the situation of a vector bundle
that is in general non-flat and non-trivial.
\section{Main Results}
\label{Sect:Res}

\subsection{Bounded holomorphic functional calculi and Kato square root type estimates}
 
In this section we 
to first illustrate how to reduce
the main quadratic estimate 
to a simpler, local quadratic estimate.
Then, we present the main theorem of this paper
and illustrate its main corollary; a Kato square
root type estimate.

We begin with the following 
adaptation of Proposition 5.2 in \cite{Morris3}.

\begin{proposition}
\label{Prop:Ass:Main}
Suppose that $(\Gamma, B_1, B_2)$ satisfies 
the hypotheses (H1)-(H3) along with
$\norm{u} \lesssim \norm{\Pi u}$
for $u \in \ran(\Pi)$, 
and that
there exists $c > 0$ and $t_0 > 0$ such that
\begin{equation*}\label{quad} \int_{0}^{t_0} \norm{\Theta_t^B P_t u}^2\ \frac{dt}{t}
	\leq c \norm{u}^2 \tag{Q1}
\end{equation*}
for all $u \in \ran(\Gamma)$,
together with three similar estimates
obtained by replacing $(\Gamma, B_1, B_2)$
by $(\adj{\Gamma},B_2,B_1)$,
$(\adj{\Gamma},\adj{B_2}, \adj{B_1})$
and $(\Gamma, \adj{B_1}, \adj{B_2})$.
Then, $\Pi_B$ satisfies
$$
\int_{0}^\infty \norm{Q_t^B u}^2\ \frac{dt}{t}
	\simeq \norm{u}^2$$ 
for all $u \in \close{\ran(\Pi_B)} \subset \Hil$.
Thus, $\Pi_B$ has a bounded holomorphic 
functional calculus.
\end{proposition}

\begin{proof}
The proof is similar to the proof
of Proposition 5.2 in \cite{Morris3}.
The assumption that
$\norm{u} \lesssim \norm{\Pi u}$
allows us to handle the integral from $t_0$ to 
$\infty$.
\end{proof}

We use the entire list 
of hypotheses (H1)-(H8) in \S\ref{Sect:Ass}
to show that the assumptions
of Proposition \ref{Prop:Ass:Main}
are satisfied.
Thus, this yields the
main theorem of this paper.

\begin{theorem}
\label{Thm:Ass:Main}
Suppose that $\cM$, $\cV$ 
and  $(\Gamma,B_1,B_2)$
satisfy (H1)-(H8). Then,
$\Pi_B$ satisfies the quadratic
estimate 
$$
\int_{0}^\infty \norm{Q_t^B u}^2\ \frac{dt}{t}
	\simeq \norm{u}^2$$ 
for all $u \in \close{\ran(\Pi_B)} \subset \Lp{2}(\cV)$
and hence has a bounded holomorphic functional 
calculus.
\end{theorem}

We defer the proof to  \S\ref{Sec:Harm}.

In particular, the projections $\chi^{\pm}(\Pi_B)$ are bounded, 
as is the operator $\sgn(\Pi_B)= \chi^{+}(\Pi_B) - \chi^{-}(\Pi_B)$. 
Thus we have the following corollary.

\begin{corollary}[Kato square root type estimate] 
Under the hypotheses of Theorem \ref{Thm:Ass:Main},
\label{Cor:Ass:Main}
\begin{enumerate}[(i)]
\item there is a spectral decomposition  
	$$\Lp{2}(\cV) = \nul(\Pi_B) \oplus E_B^{+} \oplus E_B^{-}$$
	where $E_B^{\pm} = \ran(\chi^{\pm}(\Pi_B))$
	(the direct sum is in general non-orthogonal), and 
\item 	$\dom(\Gamma) \intersect \dom(\adj{\Gamma}_B) = \dom(\Pi_B) = \dom(\sqrt{{\Pi_B}^2})$
	with
	$$\norm{\Gamma u} + \norm{\adj{\Gamma}_B u} 
		\simeq \norm{\Pi_B u} 	
		\simeq \|\sqrt{{\Pi_B}^2}u\|$$
	for all $u \in \dom(\Pi_B)$.
\end{enumerate} 
\end{corollary} 
To prove part (ii), we use the identities $  \Pi_B u = \sgn(\Pi_B)\sqrt{{\Pi_B}^2}u$ 
and $\sqrt{{\Pi_B}^2}u = \sgn(\Pi_B) \Pi_B u$ for all $u \in \dom(\Pi_B)$, 
together with the bound on  $\sgn(\Pi_B)$. 

\subsection{Stability of perturbations} 
\label{Sect:Res:Stab}

It is a consequence of the fact that the estimates 
in Theorem \ref{Thm:Ass:Main} hold for a class of operators 
with complex measurable coefficients
$B_i$, that  operators such as $\sgn(\Pi_B)$ 
are also stable under small perturbations in $B$.
 
We provide the following adaption of Theorem 6.4 in
\cite{AKMC} with a minor modification. 
In the  following theorem,
$\Hil$ denotes an abstract Hilbert
space.
We say that a 
family $\set{T(\zeta)}_{\zeta\in U}$  
of $\omega$-bisectorial operators 
has a \emph{uniformly bounded holomorphic
functional calculus} if each operator $T(\zeta)$ has a bounded 
holomorphic functional calculus on the same sector $S^o_\mu\union \set{0}$
with a bound which is uniform in $\zeta\in U$. 

\begin{theorem}[Holomorphic dependence]
\label{Thm:Res:HolDep}
Let $U \subset \C$ be an open set 
and $B_1, B_2: U \to \bddlf(\Hil)$
be holomorphic and suppose
that $(\Gamma,B_1(\zeta),B_2(\zeta))$
satisfy (H1)-(H3) uniformly for
all $\zeta \in U$.
Suppose further that
$\Pi_{B(\zeta)}$
has a uniformly bounded holomorphic functional 
calculus  on $S_{\mu}^o\cup\{0\}$ for some $\omega < \mu < \frac{\pi}{2}$
(where $\omega$ is the angle of sectoriality). Let $f \in \Hol^\infty(S_{\mu}^o)$. Then the map $\zeta \mapsto f(\Pi_{B(\zeta)})$
is holomorphic  on $U$. 
\end{theorem}

\begin{proof} This
 claim is proved in a similar way to the
first part of Theorem 6.4 in \cite{AKMC},
with the exception that instead of
invoking Theorem 2.10 in \cite{AKMC},
we note that the uniformly bounded 
holomorphic functional 
calculus assumption  is sufficient.
\end{proof}

 Next, consider the situation of $\Hil = \Lp{2}(\cV)$.
Adapting the construction of \cite{AKMC},
we define the following Hilbert space
$$\mathcal{K} = \Lp{2}(\cM \times (0,\infty), \cV; \frac{d\mu dt}{t}).$$
Then, for $\psi \in \Psi(S_{\mu}^o)$,
$t > 0$ and almost all $x \in \cM$, we
define the  operator 
$\Sq_{B(\zeta)}(\psi): \Hil \to \mathcal{K}$ by
$(\Sq_{B(\zeta)}(\psi)u)(x,t)  = \psi(t\Pi_{B(\zeta)})u(x)$.

\begin{theorem}
Under the hypothesis of Theorem \ref{Thm:Res:HolDep},
and the additional assumption that $\Hil = \Lp{2}(\cV)$,
whenever
$\omega < \mu < \frac{\pi}{2}$
(where $\omega$ is the angle of sectoriality),
the map $\zeta \mapsto \Sq_{B(\zeta)}(\psi)$
is holomorphic on $U$ for all $\psi \in \Psi(S_{\mu}^o)$.
\end{theorem}
\begin{proof}
We note that our choice of $\mathcal{K}$
is an adequate replacement to $\mathcal{K}$
in the proof of Theorem 6.4 in \cite{AKMC}.
Also, for $t > 0$, the function $\psi_t(\zeta) = \psi(t\zeta) \in \Psi(S_{\mu}^o)$
and $\norm{\psi_t}_\infty = \norm{\psi}_\infty$.
Therefore, the uniformly bounded 
holomorphic functional calculus
assumption holds uniformly in $t > 0$ and 
is again an adequate substitution
to run the rest of the argument 
of the proof of Theorem 6.4 in \cite{AKMC}.
\end{proof}

\begin{corollary}
\label{Cor:Res:Stab}
Let $\Hil, \Gamma, B_1, B_2, \kappa_1, \kappa_2$ satisfy
(H1)-(H8) and take $\eta_i < \kappa_i$. 
Set $0 < \hat{\omega}_i < \frac{\pi}{2}$ by 
$\cos\hat{\omega}_i = \frac{ \kappa_i - \eta_i}{\norm{B_i}_\infty + \eta_i}$
and $\hat{\omega} = \frac{1}{2}(\hat{\omega}_1 + \hat{\omega}_2)$.
Let $A_i \in \Lp{\infty}(\bddlf(\cV))$ satisfy
\begin{enumerate}[(i)]
\item $\norm{A_i}_\infty \leq \eta_i$,
\item $A_1 A_2 \ran(\Gamma), B_1 A_2 \ran(\Gamma),
	A_1 B_2 \ran(\Gamma) \subset \nul(\Gamma)$, and 
\item $A_2 A_1 \ran(\adj{\Gamma}), B_2 A_1 \ran(\adj{\Gamma}),
	A_2 B_1 \ran(\adj{\Gamma}) \subset \nul(\adj{\Gamma})$.
\end{enumerate}
Letting $\hat{\omega} < \mu < \frac{\pi}{2}$,
we have:
\begin{enumerate}[(i)]
\item for all $f \in \Hol^\infty(S_{\mu}^o)$,
	$$\norm{f(\Pi_B) - f(\Pi_{B+A})}
	\lesssim (\norm{A_1}_\infty + \norm{A_2}_\infty)
	\norm{f}_\infty,\qquad\text{and}$$
\item for all $\psi \in \Psi(S_{\mu}^o)$,
	$$\int_{0}^\infty \norm{\psi(t\Pi_{B})u - \psi(t\Pi_{B + A})u}^2 \frac{dt}{t}
	\lesssim (\norm{A_1}_\infty^2 + \norm{A_2}_\infty^2)
	\norm{u},$$
	whenever $u \in \Hil$.
\end{enumerate}  
The implicit constants depend on (H1)-(H8) and $\eta_i$. 
\end{corollary}
\begin{proof}
The argument proceeds in a similar way to the 
proof of Theorem 6.5 in \cite{AKMC}.
The conditions on $A_i$ guarantee
that $\tilde{B}_i(\zeta) = (B_i + \zeta A_i)$
satisfies (H3). 
This is, in fact, 
a necessary amendment to 
the original proof 
of Theorem 6.5 in \cite{AKMC}.
\end{proof} 
\section{\Poincare Inequalities}

 In this short section we  show that,  under 
appropriate geometric conditions, 
we can bootstrap the \Poincare inequality
on functions to the dyadic \Poincare inequality
on the bundle. As in \cite{Morris3}, we make
the following definition.

\begin{definition}[Local \Poincare inequality]
\index{Local \Poincare inequality}
We say that $\cM$ satisfies a \emph{local \Poincare inequality}
if there exists $c \geq 1$ such that
for all $f \in \SobH{1}(\cM)$, 
\begin{equation*}
\norm{f - f_{B}}_{\Lp{2}(B)} \leq c\ \rad(B) \norm{f}_{\SobH{1}(B)}
\tag{$\text{P}_{\text{loc}}$}
\label{Def:Pre:Ploc}
\end{equation*}
for all balls $B$ in $\cM$
such that $\rad(B) \leq 1$.
\end{definition}

\begin{remark}
Note that we allow the Sobolev norm
$\norm{\mdot}_{\SobH{1}(B)}$ over the ball on the right
of the inequality, 
rather than simply $\norm{\conn \mdot}$.
\end{remark}

The following proposition then illustrates
that under appropriate
gradient bounds on the 
GBG coordinate basis,
we can obtain 
a dyadic \Poincare inequality in
the bundle.

\begin{proposition}
\label{Prop:App:DyPoin}
Suppose that $\cM$ satisfies 
both (\ref{Def:Pre:Eloc}) and
(\ref{Def:Pre:Ploc}). 
Furthermore, suppose that
that there exists $C_G > 0$
such that in each GBG chart
with basis denoted by $\set{e^i}$
we have $\modulus{\conn e^i} \leq C_G$
for each $i$. Then,
for all $u \in \dom(\conn) = \SobH{1}(\cV)$, $t \leq \scale$,  $\Q \in \DyQ_t$, and $r \geq \frac{C_1}{\delta}$, 
$$
\int_{B} \modulus{u - u_{\Q}}^2\ d\mu \lesssim 
	(1 + e^{\lambda rt})(rt)^2 \int_{B} (\modulus{\conn u}^2 + \modulus{u}^2)\ d\mu$$
where  $B = B(x_{\Q},rt)$.
\end{proposition}
\begin{proof}
First, consider the case $(rt) \geq \frac{\brad}{5}$.
Then
$$\norm{u - u_{\Q}}_{\Lp{2}(B)}^2
	\lesssim \norm{u}^2_{\Lp{2}(B)} + \norm{u_{\Q}}^2_{\Lp{2}(B)}.$$
Recalling that $\parent{\Q}$ is the GBG cube of $\Q$, 
\begin{multline*}
\int_{B} \modulus{u_{\Q}}^2\ d\mu
	= \int_{B} \modulus{ \cbrac{\fint_{\Q} u_i\ d\mu}\ind{B(x_{\parent{\Q}},\brad)} e^i}^2\ d\mu
	\simeq \sum_{i} \int_{B} \modulus{\fint_{\Q} u_i\ d\mu}^2\ind{B(x_{\parent{\Q}},\brad)} d\mu \\
	\leq \sum_{i} \int_{B} \cbrac{ \fint_{\Q} \modulus{u_i}^2\ d\mu}\ d\mu
	\leq \frac{\mu(B)}{\mu(\Q)} \int_{B} \modulus{u}^2\ d\mu
	\lesssim (1 + r^\kappa e^{\lambda rt}) \norm{u}_{\Lp{2}(B)}^2.
\end{multline*}
Thus,
$
\norm{u - u_{\Q}}^2_{\Lp{2}(B)} 
	\lesssim (1 + r^{\kappa} e^{\lambda rt}) (rt)^2 \norm{u}^2_{\Lp{2}(B)}$
since $rt \geq \frac{\brad}{5}$.

Next, suppose that $(rt) < \frac{\brad}{5}$.
 It is easy to see
that we have $B(x_{\Q},rt) \subset B(x_{\parent{\Q}},\brad)$
and so, 
\begin{multline*}
\int_{B} \modulus{u - u_{\Q}}^2\ d\mu
	\simeq \sum_{i} \int_{B} 
		\modulus{u_i - \cbrac{\fint_{\Q}u_i\ d\mu}\ind{B(x_{\parent{\Q}},\brad)}}^2\ d\mu \\
	\leq \sum_{i} \int_{B} \modulus{u_i - (u_i)_B}^2\ d\mu 
		+ \sum_{i} \int_{B} \modulus{(u_i)_B - \cbrac{\fint_{\Q} u_i\ d\mu}}^2\ d\mu
\end{multline*}
For the first term, we invoke
(\ref{Def:Pre:Ploc}) so that
$$
\sum_{i} \int_{B} \modulus{u_i - (u_i)_B}^2\ d\mu
	\lesssim (rt)^2 \int_{B} 
		\cbrac{\sum_{i} \modulus{\conn u_i}^2  + \modulus{u}^2}\ d\mu.$$
Now, for the second term,
\begin{multline*}
\sum_{i} \int_{B} \modulus{(u_i)_B - \cbrac{\fint_{\Q} u_i\ d\mu}}^2\ d\mu
	\leq \sum_{i} \frac{\mu(B)}{\mu(\Q)}  \int_{\Q} \modulus{u_i - (u_i)_B}^2\ d\mu \\
	\lesssim (1+ r^\kappa e^{\lambda rt}) (rt)^2 \int_{B} \cbrac{\sum_{i} \modulus{\conn u_i}^2 + \modulus{u}^2}\ d\mu
\end{multline*}

Next, we note that
$
\conn u = \conn(u_i) \tensor e^i + u_i \tensor \conn e^i$
and therefore, by the hypothesis $\modulus{\conn e^i} \leq C_{G}$,
$$\sum_{i} \modulus{\conn u_i}^2 \simeq
	\modulus{\conn u_i \tensor e^i}^2 
	\leq \modulus{\conn{u}}^2 + \modulus{u_i}^2 \modulus{\conn e^i}^2
	\lesssim \modulus{\conn{u}}^2 + \modulus{u}.$$
The proof is complete by combining these estimates.
\end{proof}
\section{Kato Square Root Estimates for Elliptic Operators}
\label{Sect:App}

\subsection{The Kato square root problem on vector bundles}

Here, we present the main applications
of Theorem \ref{Thm:Ass:Main} to 
uniformly elliptic operators which arise naturally
from a connection over a vector
bundle.  First, we describe a setup of
operators which is a generalisation
of \S1 of \cite{Morris3} (and before that from \cite{AKM2}),
making the necessary changes 
to facilitate the fact that we are
working, in general, on a non-trivial bundle.

Let $\Hil = \Lp{2}(\cV) \oplus \Lp{2}(\cV) \oplus \Lp{2}(\cotanb \cM \tensor \cV)$. 
As discussed in \S\ref{prelim}, $\conn: \SobH{1}(\cV) \subset \Lp{2}(\cV) \to \Lp{2}(\cotanb \cM \tensor \cV)$ is a closed densely defined operator, and so has a well defined adjoint $\adj{\conn}$, which we denote by $ -\divv: \dom(\divv) \subset \Lp{2}(\cotanb\cM \tensor \cV) \to \Lp{2}(\cV)$.

The reason for this notation is because when the the connection
$\conn$ and the  metric $\mh$ are compatible, then $\adj{\grad}$ has the form of a divergence in the  weak sense of Proposition \ref{Prop:App:Comp} below. First some notation.

For $v \in \Ck{\infty}(\cotanb\cM \tensor \cV)$, define
$\tr \conn v$ by contracting the first two indices of 
$\conn v \in \Ck{\infty}(\cotanb\cM \tensor \cotanb\cM \tensor \cV)$
over $\mg$, to yield
a section $\tr \conn v \in \Ck{\infty}(\cV)$. 

The connection
$\conn$ and the metric $\mh$ are \emph{compatible}
if the product rule
$$X(\mh(Y,Z)) = \mh(\conn[X]Y, Z) + \mh(Y, \conn[X]Z)$$
is satisfied for every $X \in \Ck{\infty}(\tanb \cM)$ and 
$Y,Z \in \Ck{\infty}(\cV)$. 
For such connection and metric pairs,  we have 
  $\divv=-\adj{\grad} =  \tr\conn$ in the following weak sense. 

\begin{proposition}
\label{Prop:App:Comp}
Suppose that the connection $\conn$ and the metric $\mh$   are  compatible.
Then, for all $T \in \Ck[c]{\infty}(\cV)$ 
and $P \in \Ck{\infty}(\cotanb \cM\tensor \cV)$,
$$
\int_{\cM} \inprod{ \conn T, P}\ d\mu = \int_{\cM} \inprod{T, -\tr \conn P}\ d\mu.$$
\end{proposition}
\begin{proof}
Fix $x \in \cM$ so that we can locally write
$T = T_i e^i$ and $P = P_{kl} dx^k \tensor e^l$.
Next, define the $\cV$ ``inner product'' yielding a $1$-form by
$\inprod{T,P}_{\cV} = T_i P_{kl} \mh(e^i, e^l) dx^k.$

Now, to make calculations easier, further assume that
$\set{x^i}$ are normal coordinates at $x$. Then,
for any $X = X_k dx^k$, we have that the 
divergence at $x$ is $\divv X = \partial_k X_k$.
Thus, at $x$, 
\begin{align*}
\divv\inprod{T,P}_{\cV}
	= \sum_{k} &(\partial_k T_i) P_{kl} \mh(e^i, e^l) 
		+ \sum_{k} T_i P_{kl} \mh(\conn[\partial_k] e^i, e^l) \\
			&+ \sum_{k} T_i (\partial_k P_{kl}) \mh(e^i, e^l)
			+ \sum_{k} T_i P_{kl} \mh(e^i, \conn[\partial_k] e^l)
\end{align*}
by the compatibility of $\conn$ and $\mh$.

A calculation at $x$ then shows that 
$\conn T = \conn(T_i e^i) 
	= \partial_k T_i dx^k \tensor e^i + T_i dx^k \tensor \conn[\partial_k]e^i.$
Also, 
$\tr \conn P 
	= \sum_{k} \partial_k P_{kl} e^l + \sum_{k} P_{kl} \conn e^l,$
since we assumed normal coordinates at $x$, making $\conn dx^k = 0$.
Then, a direct calculation shows that 
\begin{align*}
\inprod{\conn T, P} 
	&= \sum_{k} (\partial_k T_i) P_{kl} \mh(e^i, e^l) 
		+ \sum_{k} T_i P_{kl} \mh(\conn[\partial_k] e^i, e^l),
\quad\text{and}\\
\inprod{T, \tr \conn P}
	&=  \sum_{k} T_i (\partial_k P_{kl}) \mh(e^i, e^l)
		+ \sum_{k} T_i P_{kl} \mh(e^i, \conn[\partial_k] e^l).
\end{align*}
Thus, at $x$,
$\divv \inprod{T,P}_{\cV} = \inprod{\conn T, P} + \inprod{T, \tr \conn P}.$

 By the compactness of $\spt T$,
it is easy to see that 
$\spt \inprod{T,P}_\cV,\ \spt \inprod{T, \tr \conn P}$ and  $\spt \inprod{\conn T, P}$ 
are all compact. Thus, we integrate this
equation over $\cM$ and
apply the divergence theorem to obtain the conclusion.
\end{proof}

We pause to introduce some notation. When
 $\cW,\ \tilde\cW$ are two vector bundles, define the new vector bundle $\bddlf(\cW, \tilde\cW)$
to mean the space of all maps $C: \cW \to \tilde\cW$ such that
for each $x \in \cM$, 
$C(x) \in \bddlf(\cW_x,\tilde\cW_x)$.
This is consistent with the previous
notation since 
$\bddlf(\cW) = \bddlf(\cW, \cW)$. 

With this notation in mind, 
let
$A_{00} \in \Lp{\infty}(\bddlf(\cV)),\ 
A_{01} \in \Lp{\infty}(\bddlf(\cotanb \cM \tensor \cV, \cV)),\ 
A_{10} \in \Lp{\infty}(\bddlf(\cV, \cotanb\cM \tensor \cV)),\ 
A_{11} \in \Lp{\infty}(\bddlf(\cotanb\cM \tensor \cV)).$
Then, define $A \in \Lp{\infty}(\bddlf(\cV \oplus (\cotanb\cM \tensor \cV)))$
by
$$
A = \begin{pmatrix} A_{00} & A_{01} \\ A_{10} & A_{11}\end{pmatrix}.$$
Furthermore,
let $a \in \Lp{\infty}(\bddlf(\cV))$. 
Set $B_1,B_2:\Hil\to \Hil$ by 
$$
B_1 = \begin{pmatrix} a & 0 \\ 0 & 0\end{pmatrix}
\quad\text{and}\quad
B_2 = \begin{pmatrix} 0 & 0 \\ 0 & A\end{pmatrix}\ .$$

Moreover, set
$$ S = \begin{pmatrix}I \\ \conn \end{pmatrix},\ 
\adj{S} = \begin{pmatrix} I & -\divv  \end{pmatrix},\  
\Gamma = \begin{pmatrix}0 & 0 \\ S & 0\end{pmatrix},\ 
\text{and}\ 
\adj{\Gamma} = \begin{pmatrix}0 & \adj{S} \\ 0 & 0\end{pmatrix}$$
and define the following divergence form 
operator $\Div_{A}: \dom(\Div_A) \subset \Lp{2}(\cV) \to \Lp{2}(\cV)$
by 
$$
\Div_{A}u = a\adj{S}ASu 
	= -a \divv (A_{11} \conn u) - a \divv(A_{10}u) 
	+ a A_{01} \conn u + a A_{00} u.$$

We apply Theorem \ref{Thm:Ass:Main} to prove the
Kato square root problem on vector bundles.

\begin{theorem}[Kato square root problem for vector bundles]
\label{Thm:App:KatoVB}
Suppose that $\cM$ satisfies
(\ref{Def:Pre:Eloc})
and both $\cV$ and $\cotanb \cM$
have generalised bounded geometry
(so that they are equipped with GBG coordinate systems),
and that
\begin{enumerate}[(i)]

\item 
$\cM$ satisfies (\ref{Def:Pre:Ploc}),

\item 
the GBG charts for $\cotanb\cM$ are induced
from coordinate systems on $\cM$,

\item
the connection $\conn$ and the metric $\mh$ are compatible,

\item 
there exists $C > 0$ such that
in each GBG chart $\set{e^j}$ 
for $\cV$ and $\set{dx^i}$ for $\cotanb \cM$,
we have that
$\modulus{\conn e^j},  \modulus{\partial_k \mh^{ij}}, 
\modulus{\partial_k \mg^{ij}} \leq C$ a.e., 

\item 
there exist $\kappa_1, \kappa_2 > 0$ such that
$$
\re\inprod{a u, u}  \geq \kappa_1 \norm{u}^2
\qquad\text{and}\qquad 
\re\inprod{A Sv, Sv} \geq \kappa_2 \norm{v}_{\SobH{1}}^2$$
for all $u \in \Lp{2}(\cV)$
and $v \in \SobH{1}(\cV)$, and

\item
we have that $\dom(\Lap) \subset \SobH{2}(\cV)$, and 
there exist $C' > 0$ such that
$$ \norm{\conn^2 u} \leq C' \norm{(I + \Lap)u}$$
whenever $u \in \dom(\Lap)$.
\end{enumerate} 
Then, 
\begin{enumerate}[(i)]
\item $\Pi_B$ has a bounded holomorphic functional calculus, and
\item $\dom(\sqrt{\Div_A}) = \dom(\conn) = \SobH{1}(\cV)$
 with  $\|{\sqrt{\Div_A} u}\|\simeq \norm{\conn u} + \norm{u} = \norm{u}_{\SobH{1}}$
for all $u \in \SobH{1}(\cV)$.
\end{enumerate} 
\end{theorem}
\begin{proof}[Proof of Theorem \ref{Thm:App:KatoVB}  (assuming Theorem \ref{Thm:Ass:Main})]
We show that $(\Gamma,B_1,B_2)$ satisfy
(H1)-(H8) of \S\ref{Sect:Ass} in order to invoke Theorem \ref{Thm:Ass:Main}.

Nilpotency of $\Gamma$ is immediate.
That $\Gamma$ is
densely-defined and closed
follows easily from Proposition \ref{Prop:Pre:ConnClosed}.
This settles (H1). Also, (H3)
is an easy calculation and (H4)-(H5)
are immediate. The
fact that (H6) is satisfied
is an immediate consequence of the 
Leibniz property of the connection.
The conditions (i) and (ii)
allow us to invoke Proposition \ref{Prop:App:DyPoin},
thus proving (H8)-1.
It remains to demonstrate (H2), (H7), 
and (H8)-2 hold  with 
$\Xi:\dom(\Xi) \subset \Lp{2}(\cV) \oplus  \Lp{2}(\cV) \oplus \Lp{2}(\cotanb \cM \tensor \cV)
\to \Lp{2}(\cotanb\cM \tensor \cV) \oplus  \Lp{2}(\cotanb\cM \tensor \cV) 
\oplus \Lp{2}(\Tensors[0,2]\cM \tensor \cV)$ defined by
$\Xi(u_1,u_2,u_3) = (\conn u_1, \conn u_2, \conn u_3)$.
The domain of $\Xi$ is $\dom(\Xi) = \SobH{1}(\cV) \oplus \SobH{1}(\cV) 
\oplus \SobH{1}(\cotanb\cM \tensor \cV)$. 

Fix $u \in \ran(\adj{\Gamma})$.
That is, $u = (\adj{S} v, 0)$ for
$v \in \Lp{2}(\cV) \oplus \Lp{2}(\cotanb M \tensor \cV)$.
Thus, 
$$
\re\inprod{B_1 u, u}
	= \re \inprod{a \adj{S}v, \adj{S}v}
	\geq \kappa_1 \norm{\adj{S}v}^2
	= \kappa_1 \norm{u}^2.$$
Next, let $u \in \ran(\Gamma)$.
Thus, $u = (0,Sv)$
for $v \in \Lp{2}(\cV)$.
Therefore, 
$$
\re\inprod{B_2 u, u}
	= \re\inprod{A Sv, Sv}
	\geq \kappa_2 \norm{u}^2$$
which settles (H2).

We verify (H7). 
First, let $u = (u_1, u_2, u_3) \in \dom(\Gamma)$
with $\spt u \subset \Q$
and $v = (v_1,v_2,v_3) \in \dom(\adj{\Gamma})$
with $\spt v \subset \Q$.
Then, $\Gamma u = (0, Su_1) = (0,u_1, \conn u_1)$,
$\adj{\Gamma}v = (\adj{S}(v_2,v_3),0) = (v_2 - \divv v_3,0,0)$
and we have that 
$$
\modulus{\int_{\Q} \Gamma u\ d\mu}
	= \modulus{\int_{\Q} u_1\ d\mu} + \modulus{\int_{\Q} \grad u_1\ d\mu}$$
and
$$
\modulus{\int_{\Q} \adj{\Gamma} v\ d\mu}
	= \modulus{\int_{\Q} {v_2 - \divv v_3}\ d\mu}
	\leq \modulus{\int_{\Q}  v_2\ d\mu} + \modulus{\int_{\Q} \divv v_3\ d\mu}.$$
By Cauchy-Schwartz, 
$$\modulus{\int_{\Q} u_1\ d\mu} 
	\lesssim \mu(\Q)^{\frac{1}{2}} \norm{u_1}
	\leq  \mu(\Q)^{\frac{1}{2}} \norm{u},$$
and by a similar computation,
$$\modulus{\int_{\Q}  v_2\ d\mu} \lesssim \mu(\Q)^{\frac{1}{2}} \norm{v}.$$ 
To conveniently deal with the two remaining estimates, we omit the indices
in $u_1$ and $v_3$ and note that it remains to prove 
$$  \text{(a)} \quad
\modulus{\int_{\Q} \conn u} \lesssim  \mu(\Q)^{\frac{1}{2}} \norm{u}
\quad\text{and}\quad 
\text{(b)} \quad
\modulus{\int_{\Q} \divv v} \lesssim  \mu(\Q)^{\frac{1}{2}} \norm{v}$$
for all $u \in \dom(\conn)$, $v \in \dom(\divv)$ 
with $\spt u,\ \spt v \subset \Q$.

Before continuing, we remark that 
every function $f \in \Lp[loc]{1}(\cV)$ can be written as $f= f_ie^i = \mh(f, \mh_{ki} e^k)e^i$
in $B(x_{\parent{\Q}},\brad)$. Here $\mh_{ij} = \mh(e_i, e_j)$, where  $\{e_i\}$ is 
the dual basis of $\{e^i\}$, and 
we use the same notation $\mh$ to denote the 
induced inner product on $\adj{\cV}$
by requiring that $\mh_{ij}\mh^{jk} = \delta_{i}^k$. We also remark that every function
 $F \in \Lp[loc]{1}(\cotanb \cM \tensor \cV)$ can be written as 
$F = F_{ij}\ dx^i \tensor e^j =  \mg \tensor \mh(F, \mg_{ai}h_{bj}\ dx^a \tensor e^b)\ dx^i \tensor e^j$ in $B(x_{\parent{\Q}},\brad)$ where $\mg_{ij} = \mg(\partial_{i},\partial_j)$ 
on $\tanb\cM$.

Turning to the proof of (a), let $u\in  \SobH{1}(\cV)$ with $\spt u\subset \Q$.
Choose $\psi \in \Ck[c]{\infty}(\cM)$ such that $\spt\psi\subset B(x_{\parent{\Q}},\brad)$ and 
$\psi = 1$ on $Q$, and
extend
$\psi \mg_{ai}\mh_{bj}\ dx^a\tensor e^b$ to be zero
outside of $B(x_{\parent{\Q}},\brad)$. Then, by the above remark with
 $F = \conn u$, we have the following identity on $Q$.
\begin{align*}
\int_{\Q} \conn u
	&= \int_{\Q} \mg \tensor \mh(\conn u, \psi \mg_{ai}\mh_{bj}\ dx^a\tensor e^b)\ d\mu\ dx^i \tensor e^j \\
	&= \int_{\cM} \mg \tensor \mh(\conn u, \psi \mg_{ai}\mh_{bj}\ dx^a\tensor e^b)\ d\mu\ dx^i \tensor e^j \\
	&= \int_{\cM} \mh(u, -\tr\conn (\psi \mg_{ai}\mh_{bj}\ dx^a\tensor e^b))\ d\mu\ dx^i \tensor e^j  \\
	&= \int_{\Q} \mh(u, -\tr \conn(\mg_{ai}\mh_{bj}\ dx^a\tensor e^b))\ d\mu\ dx^i \tensor e^j \\
	&= \int_{\Q} -\mh(u, \mg_{ai}\mh_{bj}\tr \conn (dx^a \tensor e^b) 
		+ \tr ((\conn \mg_{ai}\mh_{bj}) \tensor dx^a \tensor e^b))\ d\mu\ dx^i \tensor e^j\ .
\end{align*} 
We have used Proposition \ref{Prop:App:Comp} (since $\mg_{ai}\mh_{bj}\ dx^a \tensor e^b$
are smooth), the product rule for $\conn$, and the linearity of $\tr$.  We note that by an easy calculation, we have
 $\modulus{\tr(X)} \lesssim  \modulus{X}$ 
for all $x \in \cM$ whenever $X \in \Ck{\infty}(\cotanb\cM \tensor\cotanb\cM \tensor \cV)$.  
Furthermore,
the bound on the metric in each GBG chart implies bounds on
$\modulus{\mh_{ai}}$ and $\modulus{\mg_{ai}}$,
and the bounds in (iv) imply bounds on $\modulus{\partial_k \mh_{ai}}$
and $\modulus{\partial_k \mg_{bj}}$.
Since we assumed the connection to be Levi-Cevita,
we can write $\conn dx^a$ purely in terms of the Christoffel
symbols, which in turn can be written in terms of 
$\mg_{ij}$, $\mg^{ij}$ and $\partial_k \mg_{ij}$.
Also, $\modulus{e^b}$ and $\modulus{dx^a}$ are bounded
by the GBG hypothesis, $\modulus{\conn e^b}$
by (iv), and so we conclude that
$$\modulus{\mg_{ai}\mh_{bj}\tr \conn (dx^a \tensor e^b) 
		+ \tr ((\conn \mg_{ai}\mh_{bj}) \tensor dx^a \tensor e^b)} \lesssim 1.$$
On combining these estimates, and applying the Cauchy-Schwartz inequality,
we conclude that
$$\modulus{\int_{\Q} \conn u} \lesssim \mu(\Q)^{\frac{1}{2}} \norm{u}$$  as required.

To verify (b),  let $v \in \dom(\divv)$ with $\spt v \subset Q$, and apply the above remark with 
$f = \divv v$ to obtain by a similar argument that
$$
\int_{\Q} \divv v 
	= \int_{\Q} \mh(\divv v, \mh_{ki}\ e^k)\ d\mu\ e^i
	= \int_{\Q} \mg \tensor \mh(v, \conn(\mh_{ki}\ e^k))\ d\mu\ e^i.$$
Reasoning as before, we have that 
$\modulus{\conn(\mh_{ki}\ e^k)}$ is bounded and by the 
Cauchy-Schwartz inequality, we conclude that 
$$\modulus{\int_{\Q} \divv v} \lesssim  \mu(\Q)^\frac{1}{2} \norm{v}$$ as required.
This completes the proof of (H7).

To show (H8) let $\Xi(u_1,u_2,u_3) = (\conn u_1, \conn u_2, \conn u_3)$.
Upon noting that 
$$\modulus{\conn (dx^i \tensor e^j)} 
	\leq \modulus{\conn dx^i} \modulus{e^j} + \modulus{\conn e^j} \modulus{dx^i} \lesssim 1,$$
we apply Proposition \ref{Prop:App:DyPoin} which proves (H8)-1.

It remains to show (H8)-2.
Fix $v = (v_1,v_2,v_3) \in \ran(\Pi) \intersect \dom(\Pi)$.
Thus, there is $u = (u_1, u_2, u_3) \in \dom(\Pi)$ such that
$v = \Pi u = (u_2 - \divv u_3, u_1, \conn u_1)$.
A calculation the shows that
$$ 
\norm{v}^2 + \norm{\Xi v}^2
	= \norm{u_1}^2 + 2 \norm{\conn u_1}^2 + \norm{\conn^2 u_1}
		+ \cbrac{ \norm{v_1}^2 + \norm{\conn v_1}^2}.$$
Also,
$$
\norm{\Pi v}^2 = \norm{(I + \Lap) u_1}^2 + \norm{v_1}^2 + \norm{\conn v_1}^2.$$
But note that
$$
\norm{(I + \Lap)u_1}^2 
	= \inprod{ (I + \Lap)u_1, (I + \Lap)u_1}
	= \norm{u_1}^2 + 2 \norm{\conn u_1}^2 + \norm{\Lap u_1}^2.$$
Combining these estimates with (iv), we have that
$\norm{v}^2 + \norm{\Xi v}^2 \lesssim \norm{\Pi v}$
which proves (H8)-2.

Now, by invoking Theorem \ref{Thm:Ass:Main},
we conclude that $\Pi_B$ has a bounded holomorphic 
functional calculus.
By its Corollary \ref{Cor:Ass:Main}, 
we obtain that $\norm{\sqrt{{\Pi_B}^2}v} \simeq \norm{\Pi_B v}$
for $v \in \dom(\Pi_B) = \dom(\sqrt{{\Pi_B}^2})$.
Fix $u \in \SobH{1}(\cV)$. Then $v = (u,0,0) \in \dom(\Pi_B)$
and
$$ \norm{\sqrt{\Div_{A}}u} = \norm{\sqrt{{\Pi_B}^2}v}
	\simeq \norm{\Pi_B v} = \norm{v}_{\SobH{1}}$$
which finishes the proof.
\end{proof} 

\begin{remark} 
Instead of taking $\conn$ to be a connection,
we could have
considered a \emph{sub-connection},
by which we  mean
a map $\conn: \Ck{\infty}(\cV) \to \Ck{\infty}(\cotanb\cM \tensor \cV)$
that is function linear
and satisfies the Leibniz property on 
a sub-bundle $\cE$ of $\cotanb\cM$, and 
vanishes outside of $\cE$. 
This is related to the study of square roots
of elliptic operators
associated with sub-Laplacians
on Lie groups \cite{BEMc}.
However, it is not clear to us whether
a sub-connection is automatically
densely-defined and closable.
\end{remark}

\subsection{Manifolds with injectivity and Ricci bounds}

In this section, we apply Theorem \ref{Thm:App:KatoVB} to establish 
the Kato square root estimates for manifolds
which have injectivity radius bounds
and Ricci bounds.
Our approach is to show that under
these conditions, there exist \emph{harmonic}
coordinates which gives us bounds
on the metric and its derivatives.
We then use these coordinates to show that 
the bundle of $(p,q)$ tensors satisfy the GBG criterion.
 
We first present the following
theorem which is really contained
in the observation 
following Theorem 1.2 in \cite{Hebey}.

\begin{theorem}[Existence of global harmonic coordinates]
Suppose there exist $\kappa, \eta > 0$
such that $\inj(M) \geq \kappa$ and
$\modulus{\Ric} \leq \eta$.
Then, for any $A > 1$ and $\alpha \in (0,1)$,
there exists $r_H = r_H(n,A,\alpha,\kappa,\eta) > 0$
such that $B(x,r_H)$ corresponds to a coordinate
system satisfying: 
\begin{enumerate}[(i)]
\item $A^{-1} \ddelta_{ij} \leq \mg_{ij} \leq A \ddelta_{ij}$
	as bilinear forms and,
\item $\begin{aligned}[t]\sum_{l} r_H
		\sup_{y \in B(x,r_H)} \modulus{\partial_{l} \mg_{ij}(y)}
		+ \sum_{l} r_H^{1 + \alpha}
		 \sup_{y \neq z} 
		\frac{\modulus{\partial_l \mg_{ij}(z) - \partial_l \mg_{ij}(y)}}
		{d(y,z)^\alpha} \leq A - 1.
	\end{aligned}$
\end{enumerate}
\end{theorem}

This immediately gives us the existence of GBG coordinates for tensor fields.

\begin{corollary}[Existence of GBG coordinates for finite rank tensors]
\label{Cor:App:GBGExt} Under the assumptions that $\inj(M) \geq \kappa>0$ and
$\modulus{\Ric} \leq \eta$, 
there exist GBG coordinates
for $\Tensors[p,q] \cM$.
Furthermore, in each basis $\set{e^i}$ in each GBG
coordinate system for $\Tensors[p,q]\cM$,
we have that $\modulus{\conn e^i} \leq C_{p,q}.$
\end{corollary}
\begin{proof}
First,  we note that the Ricci bounds imply that
there exists  $\eta \in\R$ such that 
that $\Ric \geq \eta \mg$.
Thus, as in the  proof of Theorem 1.1 in  \cite{Morris3},
we conclude that $\cM$ satisfies (\ref{Def:Pre:Eloc}).

Fix $A = 2$ and $\alpha = \frac{1}{2}$.
The previous theorem guarantees the existence
of harmonic coordinates 
for these choices.
Thus, this yields GBG coordinates for $\tanb \cM$ 
with $2^{-1} \leq \mg \leq 2$.

It is an easy calculation to show that $G \simeq I$
implies $G^{-1} \simeq I$ with the same
constants for a positive definite
matrix $G$. Thus, we obtain GBG coordinates
for $\cotanb \cM$ with $2^{-1} \leq \mg \leq 2$
where we denote the metric on $\cotanb\cM$ 
also by $\mg$.

Next, if two inner products
$u,v$ on vector spaces $U,V$ satisfy
$C_1^{-1} \leq u \leq C_1$ and
$C_2^{-1} \leq v \leq C_2$, then
$(C_1 C_2)^{-1} \leq u \tensor v \leq C_1 C_2$
on $U \tensor V$. Thus, by induction,
we have that 
$2^{-pq} \leq \mg \leq 2^{pq}$
for the metric $\mg$ on $\Tensors[p,q] \cM$.

For the gradient bounds,
first consider $\tanb \cM = \Tensors[1,0]\cM$.
Since we assume our connection is Levi-Civita,
we can write the Christoffel symbols 
$\Christ{ij}{k}$ purely 
in terms of $\partial_{a} \mg_{bc}$
and $\mg^{ab}$.
The Cauchy-Schwartz
inequality allows us to bound the
$\mg^{ab}$ and the bounds on $\partial_{a} \mg_{bc}$
comes from the theorem.
Thus, $\modulus{\Christ{ij}{k}} \leq C$
and so $\modulus{\conn e_i} \leq C$.
An inductive argument then yields
the result for $\Tensors[p,q]\cM$
with the constant dependent on $p$ and $q$.
\end{proof}

With this result, we can apply the general
Theorem \ref{Thm:App:KatoVB} to obtain
the following solution to the Kato square root problem
on finite rank tensors.

\begin{theorem}[Kato square root problem on finite rank tensors]
\label{Thm:App:KatoTen}
 Suppose that 
$\modulus{\Ric} \leq C$, 
$\inj(M) \geq \kappa > 0$,
and
the following ellipticity
condition holds: 
there exist $\kappa_1, \kappa_2 > 0$ such that
$$
\re\inprod{a u, u}  \geq \kappa_1 \norm{u}^2
\qquad\text{and}\qquad 
\re\inprod{A Sv, Sv} \geq \kappa_2 \norm{v}_{\SobH{1}}^2$$
for all $u \in \Lp{2}(\Tensors[p,q]\cM)$
and $v \in \SobH{1}(\Tensors[p,q]\cM)$.
Suppose further that $\dom(\Lap) \subset \SobH{2}(\cV)$ and 
that there exists $C' > 0$
such that
\begin{equation*}\label{Riesz} \norm{\conn^2 u} \leq C' \norm{(I + \Lap)u} \tag{R}
\end{equation*}
whenever $u \in \dom(\Lap)$.
Then, 
$\dom(\sqrt{\Div_A}) = \dom(\conn) = \SobH{1}(\Tensors[p,q]\cM)$
and $\|{\sqrt{\Div_A} u}\|\simeq \norm{\conn u} + \norm{u} = \norm{u}_{\SobH{1}}$
for all $u \in \SobH{1}(\Tensors[p,q]\cM)$.
\end{theorem}
\begin{proof}
We apply Theorem \ref{Thm:App:KatoVB}.
The Ricci bounds imply that 
there exists  $\eta \in\R$ such that 
that $\Ric \geq \eta \mg$.
This shows condition (i) 
as in the proof of  Theorem 1.1 in  \cite{Morris3}. Conditions (ii) and (iv) are
a consequence of
Corollary \ref{Cor:App:GBGExt},
and (iii) holds because the connection is Levi-Cevita. 
\end{proof} 

The Riesz transform condition \eqref{Riesz} is satisfied automatically
for $(0,0)$ tensors, or in other words, for scalar-valued functions, 
as we now show.

\begin{proof}[Proof of Theorem \ref{Thm:Int:KatoFn}]
Recall the Wietzenb\"ock-Bochner
identity
$$
\inprod{\Lap(\conn f), \conn f}
	= \frac{1}{2} \Lap(\modulus{\conn f}^2) + \modulus{\conn^2 f}^2 + \Ric(\conn f, \conn f)
$$
for all $f \in \Ck{\infty}(\cM)$.
When $f \in \Ck[c]{\infty}(\cM)$, we get
$\norm{\conn^2 f} \lesssim \norm{(I + \Lap)f}$
by integrating this identity.

The fact that $\dom(\Lap) \subset \SobH{2}(\cM)$
follows from Proposition 3.3 in \cite{Hebey}
as does the density of $\Ck[c]{\infty}(\cM)$ in $\dom(\Lap)$
Thus, $\norm{\conn^2 u} \lesssim \norm{(I + \Lap)u}$
for $u \in \dom(\Lap)$. Hence the hypotheses of Theorem \ref{Thm:App:KatoTen} hold, so Theorem \ref{Thm:Int:KatoFn} follows as a consequence. 
\end{proof}
\section{Lipschitz Estimates and Stability}

In this short section, we demonstrate
Lipschitz estimates for the functional calculus 
and the stability of the square root.

Let $\tilde{a}$ and $\tilde{A}$
satisfy the same conditions as specified for $a$ and $A$ 
prior to Theorem \ref{Thm:App:KatoVB}, and set
$$\tilde{B}_1 = \begin{pmatrix} \tilde{a} & 0 \\ 0 & 0 \end{pmatrix}
\quad\text{and}\quad
\tilde{B}_2 = \begin{pmatrix} 0 & 0 \\ 0 & \tilde{A} \end{pmatrix}.$$

On noting that $\tilde{B}_i$ satisfy conditions 
(i)-(iii) of Corollary \ref{Cor:Res:Stab}, we have the
following Lipschitz estimate.

\begin{theorem}[Lipschitz estimate]
\label{Thm:App:LipVB}
Assume the hypotheses of 
Theorem \ref{Thm:App:KatoVB}
and fix
$\eta_i < \kappa_i$. Suppose that
$\tilde{B_i}$ satisfy $\|{\tilde{B}_i\|}_\infty \leq \eta_i$
for $i = 1,2$
and set $0 < \hat{\omega}_i < \frac{\pi}{2}$ by 
$\cos\hat{\omega}_i = \frac{ \kappa_i - \eta_i}{\norm{B_i}_\infty + \eta_i}$
and $\hat{\omega} = \frac{1}{2}(\hat{\omega}_1 + \hat{\omega}_2)$.
Then, for all $\hat{\omega} < \mu < \frac{\pi}{2}$,
$$\norm{f(\Pi_B) - f(\Pi_{B+\tilde{B}})}
	\lesssim (\|\tilde{B}_1\|_\infty + \|\tilde{B}_2\|_\infty)
	\norm{f}_\infty$$
for all $f \in \Hol^\infty(S_{\mu}^o)$, and
$$\int_{0}^\infty \norm{\psi(t\Pi_{B})v - \psi(t\Pi_{B + \tilde{B}})v}^2 \frac{dt}{t}
	\lesssim (\|\tilde{B}_1\|_\infty^2 + \|\tilde{B}_2\|_\infty^2)
	\norm{v},$$
for all $\psi \in \Psi(S_{\mu}^o)$ and all $v \in \Hil$. 
The implicit constants depend in particular on on $B_i$ and $\eta_i$.
\end{theorem}

We use the coefficients $\tilde{a}$ and $\tilde{A}$
to perturb the coefficients $a$ and $A$. 
Then, we construct the following perturbed operator $\Div_{A + \tilde{A}}$
defined similar to $\Div_{A}$ given by
$\Div_{A + \tilde{A}} u = (a + \tilde{a})\adj{S}(A + \tilde{A})S u$
for $u \in \dom(\Div_{A + \tilde{A}})$.

\begin{theorem}[Stability of the square root]
\label{Thm:App:StabSq}
Assume the hypotheses of Theorem \ref{Thm:App:KatoVB} 
and fix 
$\eta_i < \kappa_i$. If  $\norm{\tilde{a}}_\infty \leq \eta_1$,
$\|\tilde{A}\|_\infty \leq \eta_2$,  then
$$\norm{\sqrt{{\Div_A}}\,u - \sqrt{{\Div_{A + \tilde{A}}}}\,u}
	\lesssim (\norm{\tilde{a}}_\infty + \|\tilde{A}\|_\infty)
	\norm{u}_{\SobH{1}}$$
for all $u \in \SobH{1}(\cV)$. 
The implicit constant depends in particular on $A, a$ and $\eta_i$.
\end{theorem}
\begin{proof}
Let $\tilde{B}_i$ be given as in the hypothesis of 
Theorem \ref{Thm:App:LipVB}, so that $\|\tilde{B}_1\|_\infty = \norm{\tilde{a}}_\infty \leq \eta_1$
and $\|\tilde{B}_2\|_\infty = \|\tilde{A}\|_\infty \leq \eta_2$. 
By Theorem \ref{Thm:App:LipVB}, 
\begin{equation*}\label{perturb}\norm{\sgn(\Pi_B)v - \sgn(\Pi_{B+\tilde{B}})v}
	\lesssim (\|\tilde{a}\|_\infty + \|\tilde{A}\|_\infty)
	\norm{v}\end{equation*}
for all $v\in \Hil$, in particular for all 
$v =\begin{pmatrix}  0 \\ Su \end{pmatrix}$ with $u \in \SobH{1}(\cV)$. Note that 
$v= \Pi_B\begin{pmatrix} u \\  0 \end{pmatrix} = \Pi_{B+\tilde B}\begin{pmatrix} u \\  0 \end{pmatrix}$
so $\sgn(\Pi_B)v = \begin{pmatrix} \sqrt{{\Div_A}}\,u \\  0 \end{pmatrix}$ 
and $\sgn(\Pi_{B+\tilde B})v = \begin{pmatrix} \sqrt{{\Div_{A+\tilde A}}}\,u \\  0 \end{pmatrix}$. 
Thus, on substitution, we obtain the desired result.
\end{proof}

We point out that the conclusions
of both these theorems 
hold if, instead of assuming the hypotheses of 
Theorem \ref{Thm:App:KatoVB},
we assume the hypotheses of Theorem \ref{Thm:App:KatoTen}.
This yields Lipschitz estimates and 
the stability result for $(p,q)$ tensors.
We
conclude this section by
highlighting the stability of the square root
for scalar-valued functions 
as a corollary.

\begin{corollary}[Stability of the square root for functions]
\label{Thm:App:Stab}
Suppose that 
$\modulus{\Ric} \leq C$, 
$\inj(M) \geq \kappa > 0$
and
the following ellipticity
condition holds: 
there exist $\kappa_1, \kappa_2 > 0$ such that
$$
\re\inprod{a u, u}  \geq \kappa_1 \norm{u}^2
\qquad\text{and}\qquad 
\re\inprod{A Sv, Sv} \geq \kappa_2 \norm{v}_{\SobH{1}}^2$$
for all $u \in \Lp{2}(\cM)$
and $v \in \SobH{1}(\cM)$.
Fix $\eta_i < \kappa_i$. If  $\norm{\tilde{a}}_\infty \leq \eta_1$,
$\|\tilde{A}\|_\infty \leq \eta_2$,  then
$$\norm{\sqrt{{\Div_A}}\,u - \sqrt{{\Div_{A + \tilde{A}}}}\,u}
	\lesssim (\norm{\tilde{a}}_\infty + \|\tilde{A}\|_\infty)
	\norm{u}_{\SobH{1}}$$
for all $u \in \SobH{1}(\cM)$. 
The implicit constant depends on $C, \kappa, \kappa_i, A, a$ and $\eta_i$.
\end{corollary}
\section{Harmonic Analysis}
\label{Sec:Harm}

 \subsection{Carleson measure reduction} 
\label{Sec:Harm:CN}

It remains for us to prove Theorem \ref{Thm:Ass:Main}. 
The main point is to show that the 
main local quadratic estimate \eqref{quad} in Proposition 
\ref{Prop:Ass:Main} is a consequence of hypotheses (H1)-(H8). 

The proof proceeds by reducing the main quadratic estimate
to a \emph{Carleson measure} estimate.
 Thus,  we first recall the notion of a (local)-Carleson measure.
Set $\cM_{+} = \cM \times (0, t_0]$,
for some $t_0 < \infty$.
We emphasise that we restrict
our considerations to $t \leq t_0$.
The \emph{Carleson box} over $\Q \in \DyQ_t$
is then defined as $\CBox_{\Q} = \close{\Q} \times (0,\len(\Q)]$.
A positive Borel measure $\nu$
 on $\cM_{+}$ is 
called a \emph{Carleson measure}
if there exists $C > 0$ such that 
${\nu(\CBox_{\Q})} \leq C \mu(\Q)$
for all dyadic cubes $\Q \in \DyQ_t$
for $t \leq t_0$. 
The Carleson norm $\norm{\nu}_{\Carl}$
is defined by 
$$
\norm{\nu}_{\Carl} = \sup_{\Q \in \DyQ_t,\ t \leq t_0} \frac{\nu(\CBox_{\Q})}{ \mu(\Q)}$$
Let $\Carl$ denote the set of all Carleson measures.

 The reader will find a more
elaborate description of Carleson measures
in the classical setting
in \S\RNum{2}.2 of \cite{stein:harm}
by Stein. The local construction described
here is a fraction of a
larger theory explored by Morris in 
\cite{Morris} and
\cite{Morris3}. 

With a description of a Carleson
measure in hand, we now illustrate how to reduce the
\emph{main local
quadratic estimate} \eqref{quad}
to a Carleson measure estimate. 
The 
key is the following consequence of Carleson's theorem,
which is a special case of Theorem 4.3 in \cite{Morris3}.
 Recall the dyadic averaging operator $\Av_t$
from \S\ref{Sect:Prelim:GBG}. 

\begin{proposition}
\label{Prop:Harm:AvCarl}
For all $u \in \Hil$ and for all $\nu \in \Carl$,
$$\iint_{\cM \times (0, t_0]} \modulus{\Av_t u}^2\ d\nu(x,t)
\lesssim \norm{u}^2 \norm{\nu}_{\Carl}.$$
\end{proposition}
 
Further, recall
that  whenever $w = w_i\ e^i(x) \in \cV_x$
for $x \in \Q$ in $\set{e^i}$, 
the associated GBG coordinates of $\Q$,
we define the the associated GBG constant section
$\omega(y) = w_i\ e^i(y)$ when $y \in B(x_{\parent{\Q}},\brad)$
and $\omega(y) = 0$ otherwise. 
Then, we define the \emph{principal part}
as $\pri_t(x)w = (\Theta_t^B \omega)(x)$. 
With this notation, and for $0 < t_0 < \infty$ to be chosen later, we 
split the main quadratic estimate in the following way:
\begin{align*}\label{quad2}
\int_{0}^{t_0} \norm{\Theta_t^B P_t u}^2\ \frac{dt}{t} 
	&\lesssim \int_{0}^{t_0} \norm{\Theta_t^B P_t u - \pri_t\Av_t P_t u}^2\ \frac{dt}{t} \tag{Q2}\\
		&+ \int_{0}^{t_0} \norm{\pri_t \Av_t(P_t - I)u}^2\ \frac{dt}{t} \\
		&+  \int_{0}^{t_0} \int_{M} \modulus{\Av_t u}^2 \modulus{\pri_t}^2\ \frac{d\mu(x)\ dt}{t}. 
\end{align*}
We call the first two terms on the right of \eqref{quad2} the \emph{principal terms}.
Proposition \ref{Prop:Harm:AvCarl}
 then 
allows us to reduce estimating the last term to 
proving that
$$A \mapsto \int_{A} \modulus{\pri_t(x)}^2 \frac{d\mu(x)\ dt}{t}$$
is a Carleson measure.
We call this term the \emph{Carleson term}.

\subsection{Estimation of principal terms}

 In this section, as the title suggests,
we illustrate how to estimate the two 
principal terms of \eqref{quad2}. 
We proceed to do so by
coupling the existence of
exponential off-diagonal bounds
with our dyadic \Poincare inequality
and cancellation hypothesis.
The estimates here are straightforward
and are more or less adapted from
\cite{AKMC}, \cite{Morris3} and \cite{AAMC}.

First, we quote the the following theorem of
\cite{Morris3}, which is essentially
contained in \cite{AKMC}.

\begin{proposition}[Off-diagonal bounds]
\label{Prop:Harm:Offdiag}
Let $U_t$ be either $R_t^B, P_t^B,Q_t^B$ or $\Theta_t^B$
for $t \in \R^+$. There exists a $C_{\Theta} > 0$,
which only depends on (H1)-(H6),
for every $M > 0$ there exists $c > 0$
with
$$ \norm{\ind{E} U_t u} \leq c \inprod{\frac{\modulus{t}}{d(E,F)}}^{M}
		\exp\cbrac{-C_{\Theta} \frac{d(E,F)}{t}} \norm{\ind{F} u}$$
whenever $E,F$ are Borel, $\spt u \subset F$.
\end{proposition}

In \cite{Morris3}, Morris points out that
$\Theta_t^B$ extends to an operator
$\Theta_t^B: \Lp{\infty}(\cV) \to \Lp[loc]{2}(\cV)$
when $t \in (0, \inprod{C_{\Theta}/(2\lambda\inprod{C_1\delta^{-1}})}]$. 
Since we require this in the harmonic analysis,
 and recalling the scale $\scale$ which we 
chose in \S\ref{Sect:Prelim:GBG} in interfacing
GBG with the dyadic decomposition,  we will 
fix an even smaller scale $\hscale \leq \scale$
such that $\hscale \leq \inprod{C_{\Theta}/(2\lambda\inprod{C_1\delta^{-1}})}$.

Next, we have the following technical lemma.

\begin{lemma}
\label{Lem:Harm:Indest}
Let $r > 0$ and suppose that $\set{B_j = B(x_j, r)}$
is a disjoint collection of balls. Then, whenever $\eta \geq 1$,
$$ 
\sum_{j} \ind{\eta B_j} \lesssim \eta^{\kappa} \e^{4\lambda \eta r }.$$\
\end{lemma}
\begin{proof}
Fix $x \in \cM$ and let 
$\sC_{x} = \set{x_j \in \cM: x \in B(x_j,\eta r)}$.
It is easy to see that 
$\sum_{j} \ind{\eta B_j}(x) = \card \sC_{x}$.
That $x_j \in \sC_{x}$ is equivalent to
saying that $d(x,x_j) < \eta r$, and therefore
for any $y \in \eta B_j$, 
$d(x,y) \leq d(x,x_j) + d(x_j, y) < (\eta + 1) r$.
That is, $B(x,(\eta + 1)r) \supset B(x_j,r)$
and by the disjointness of $\set{B_j}$,
$\mu (B(x,(\eta + 1)r)) \geq \sum_{x_j \in \sC_{x}} \mu(B(x_j,r))$.

Next, note that (\ref{Def:Pre:Eloc}) implies that
$ \mu(\eta B_j) \lesssim \eta^{\kappa} \e^{\lambda \eta r} \mu(B_j)$
and therefore,
$$
\sum_{x_j \in \sC_{x}} \mu(\eta B_j) \lesssim \eta^{\kappa} \e^{\lambda \eta r } \mu(B(x,(\eta + 1)r)).$$
Thus, it is enough to compare $\mu(\eta B_j)$ to $\mu(B(x,(\eta + 1)r))$.
So, take any $y \in B(x,(\eta + 1)r)$, and note that
$d(x,y) \leq d(x, x_j) + d(x_j, y) < (2\eta + 1)r < 3 \eta r$.
Thus,
$$\mu(B(x,(\eta + 1)r)) 
	\leq \mu(B(x_j, 3\eta r) 
	\lesssim 3^\kappa \e^{3\lambda \eta r } \mu(B(x_j,\eta r))$$
and the estimate 
$$
\card \sC_{x}\ \e^{-3\lambda \eta r} 
	\lesssim \sum_{x_j \in \sC_{x}} \frac{\mu(\eta B_j)}{\mu(B(x,(\eta+1)r))}
	\lesssim \eta^{\kappa} e^{\lambda c \eta r}$$
completes the proof.
\end{proof}

\begin{proposition}[First principal term estimate]
For all $u \in \ran(\Pi)$,
$$
\int_{0}^{t_2} \norm{\Theta_t^B P_t u - \pri_t\Av_t P_t u}^2\ \frac{dt}{t} \lesssim \norm{u}^2$$
where $t_2 \leq \min\set{\hscale,\frac{C_{\Theta}}{4\lambda (c + 4\tilde{c})}}.$
\end{proposition}
\begin{proof}
Let $v = P_t u$.
\begin{enumerate}[(i)]

\item First,
we note that  
$$\norm{\Theta_t^B P_t u - \pri_t\Av_t P_t u}^2 = 
	\sum_{\Q \in \DyQ_t} \norm{\Theta_t^B(v - v_{\Q})}_{\Lp{2}(\Q)}^2.$$
For each $\Q \in \DyQ_t$, write $B_{\Q} = B(x_{\Q}, \frac{C_1}{\delta} t) \supset \Q$
and $C_j(\Q) = 2^{j+1}B_{\Q} \setminus 2^j B_{\Q}.$
Then, for each such cube $\Q$,
\begin{align*}
\int_{\Q} \modulus{\Theta_t^B(v - v_{\Q})}^2\ d\mu
	&= \int_{\Q} \modulus{\Theta_t^B \cbrac{\sum_{j=0}^\infty \ind{C_j(\Q)} (v - v_{\Q})}}^2\ d\mu  \\
	&\leq \sum_{j=0}^\infty  \int_{\Q} \modulus{\Theta_t^B (\ind{C_j(\Q)} (v - v_{\Q}))}^2\ d\mu \\
	&\lesssim \sum_{j=0}^\infty \inprod{\frac{t}{d(\Q,C_j(\Q))}}^M 
		\exp\cbrac{- C_{\Theta} \frac{d(\Q,C_j(\Q))}{t}}
		\int_{C_j(\Q)} \modulus{v - v_{\Q}}^2\ d\mu
\end{align*}

\item
Next, note that by (4.1) in \cite{Morris3} 
$$2^j \frac{C_1}{\delta} t \leq d(x_{\Q}, C_j(\Q)) \leq d(\Q, C_j(\Q)) + \diam \Q$$
which implies that
$$
\inprod{\frac{t}{d(\Q,C_j(\Q))}}^M \lesssim 2^{-M(j+1)}.$$
Next, for $j \geq 1$, 
$$d(\Q,C_j(\Q)) \geq 2^j \frac{C_1}{2\delta} t$$
and therefore,  
$$\exp\cbrac{- C_{\Theta} \frac{d(\Q,C_j(\Q))}{t}}
	\leq \exp\cbrac{-  \frac{C_{\Theta} C_1}{4\delta} 2^{j+1}}.$$
For $j = 0$,
$$
\exp\cbrac{- C_{\Theta} \frac{d(\Q,C_j(\Q))}{t}} 
	= 1
	= \exp\cbrac{\frac{C_{\Theta} C_1}{4\delta}}
		\exp\cbrac{-  \frac{C_{\Theta} C_1}{4\delta} 2^{0}}.$$
Fix $t' > 0$ to be chosen later. Then, for all $t \leq t'$,
$$\exp\cbrac{- C_{\Theta} \frac{d(\Q,C_j(\Q))}{t}}
	\lesssim \exp\cbrac{-  \frac{C_{\Theta} C_1}{4\delta t'} 2^{j+1}t}$$
for all $j \geq 0$.

\item
Combining the estimates in (i) and (ii),
$$\int_{\Q} \modulus{\Theta_t^B(v - v_{\Q})}^2\ d\mu
	\lesssim \sum_{j=0}^\infty 2^{-M(j+1)} \exp\cbrac{-\frac{C_{\Theta} C_1}{4\delta t'} 2^{j+1}t}
		\int_{C_j(\Q)} \modulus{v - v_{\Q}}^2\ d\mu.$$

Since $v = P_t u \in \dom(\Pi^2) = \dom(\Pi)\intersect \ran(\Pi)$,
we conclude from (H8)-1 that
\begin{align*}
\int_{C_j(\Q)} \modulus{v - v_{\Q}}^2\ d\mu
	&\lesssim \cbrac{1 + \cbrac{\frac{C_1}{\delta}}^\kappa 2^{\kappa(j+1)}
			\exp\cbrac{\frac{\lambda c C_1}{\delta} 2^{j+1} t}}
			\cbrac{\frac{C_1}{\delta}}^2 2^{2 (j+1)} \\	
	&\qquad\qquad
	t^2 \int_{\tilde{c} 2^{j+1}B_{\Q}} (\modulus{\Xi v}^2 + \modulus{v}^2)\ d\mu.
\end{align*}

Therefore,
\begin{align*}
\sum_{\Q \in \DyQ_t} \int_{\Q} &\modulus{\Theta_t^B(v - v_{\Q})}^2\ d\mu \\
	&\lesssim \sum_{\Q \in \DyQ_t} \sum_{j=0}^\infty 2^{-M(j+1)} 
			\exp\cbrac{-\frac{C_{\Theta} C_1}{4\delta} 2^{j+1}} \\
			&\qquad\qquad \cbrac{1 + \cbrac{\frac{C_1}{\delta}}^\kappa 2^{\kappa(j+1)}
			\exp\cbrac{\frac{\lambda c C_1}{\delta} 2^{j+1} t}}
			\cbrac{\frac{C_1}{\delta}}^2 2^{2 (j+1)} t^2 \\
		&\qquad\qquad\int_{\tilde{c} 2^{j+1}B_{\Q}} (\modulus{\Xi v}^2 + \modulus{v}^2)\ d\mu \\
	&\lesssim \sum_{j=0}^\infty 2^{-(M-\kappa-2)(j+1)} 
		\bbrac{ \exp\cbrac{-\frac{C_{\Theta} C_1}{4\delta t'} 2^{j+1} t}
			+ 	
			 \exp\cbrac{-\frac{C_1}{\delta}\cbrac{\frac{C_{\Theta}}{4 t'}- \lambda c} 2^{j+1} t}} \\
		&\qquad\qquad t^2 \int_{\cM} \sum_{\Q \in \DyQ_t} \ind{\tilde{c} 2^{j+1}B_{\Q}}
			(\modulus{\Xi v}^2 + \modulus{v}^2)\ d\mu.
\end{align*}

\item
Set $\eta = \tilde{c} 2^{j+1}C_1/(\delta a_0)$ and $r = a_0 t$
so that $\set{B(x_{\Q}, a_0 t) \subset \Q}$ is 
disjoint, and invoke Lemma \ref{Lem:Harm:Indest}
to conclude that
$$\ind{\tilde{c} 2^{j+1}B_{\Q}} \lesssim 2^{\kappa(j+1)} \exp\cbrac{\frac{4\lambda \tilde{c} C_1}{\delta} 2^{j+1}t}.$$

Combining this with (iii), we have
\begin{align*}
\sum_{\Q \in \DyQ_t} \int_{\Q} &\modulus{\Theta_t^B(v - v_{\Q})}^2\ d\mu \\
	&\lesssim \sum_{j=0} ^\infty 2^{-(M-2\kappa - 2)(j+1)} \\
		&\qquad\qquad\bbrac{ \exp\cbrac{-\frac{C_1}{\delta}\cbrac{\frac{C_{\Theta}}{4 t'} - 4\lambda \tilde{c}} 2^{j+1} t}
			+ 	
			 \exp\cbrac{- \frac{C_1}{\delta}\cbrac{\frac{C_{\Theta}}{4 t' } - \lambda c - 4\lambda \tilde{c}} 2^{j+1} t}} \\
		&\qquad\qquad t^2 \cbrac{\norm{\Xi v}^2 + \norm{v}^2}
\end{align*}

\item
Now, we choose $t' \leq \hscale$ so that
$$
-\frac{C_1}{\delta}\cbrac{\frac{C_{\Theta}}{4 t'} - 4 \lambda \tilde{c}} \leq 0
\quad\text{and}\quad
	- \frac{C_1}{\delta}\cbrac{\frac{C_{\Theta}}{4 t' } - \lambda c - 4 \lambda \tilde{c}} \leq 0.$$
That is,
$$
\frac{C_{\Theta}}{4\lambda ( c + 4\tilde{c}) }  \leq t'.$$
We can, therefore, set $t' = t_2$ and
then,
$$\sum_{\Q \in \DyQ_t} \int_{\Q} \modulus{\Theta_t^B(v - v_{\Q})}^2\ d\mu
	\lesssim t^2 \sum_{j=0} ^\infty 2^{-(M-2\kappa -2)(j+1)} \norm{\Pi v}^2$$
by invoking (H8)-2. By choosing $M > 2\kappa + 2$, and substituting
$v = P_t u$, we have
$$ \int_{0}^{t_2} \norm{\Theta_t^B P_t u - \pri_t\Av_t P_t u}^2\ \frac{dt}{t}
	\lesssim \int_{0}^{t_2} \norm{t \Pi P_t u}^2\ \frac{dt}{t}
	\leq \int_{0}^{\infty} \norm{t \Pi P_t u}^2\ \frac{dt}{t}
	\lesssim \norm{u}.$$ 
\end{enumerate}
\end{proof}

Next, as in \cite{AKMC} and \cite{Morris3}, 
we note the following consequence of (H7).

\begin{lemma}
\label{Lem:Harm:Upsilon}
Let $\Upsilon = \Gamma,\ \adj{\Gamma}$ or $\Pi$.
Then,
$$\modulus{\fint_{\Q} \Upsilon u\ d\mu}
	\lesssim \frac{1}{\len(\Q)^\eta}
		\cbrac{\fint_{\Q} \modulus{u}^2\ d\mu}^{\frac{\eta}{2}}
		\cbrac{\fint_{\Q} \modulus{\Upsilon u}^2\ d\mu}^{1 - \frac{\eta}{2}}
		+ \fint_{\Q} \modulus{u}^2$$
for all $ u \in \dom(\Upsilon)$
and $\Q \in \DyQ_t$ where $t \leq \hscale$.
\end{lemma}

The proof of this lemma is
the same as that
of the proof of Lemma 5.9 in \cite{Morris3}.
Thus, we deduce the following.

\begin{proposition}[Second principal term estimate]
For all $u \in \Lp{2}(\cV)$, we have 
$$ \int_{0}^{\hscale} 
	\norm{\pri_t \Av_t(P_t - I)u}^2 \frac{dt}{t} \lesssim \norm{u}^2.$$
\end{proposition}

The proof of this proposition is 
similar to the proof of
Proposition 5.10 in \cite{Morris3}
with the principle difference
being that we only consider $t \leq \hscale$.

We have demonstrated how to
estimate the two principal terms of \eqref{quad2}. 

\subsection{Carleson measure estimate}
 
We are left with the task of estimating
the final term of \eqref{quad2}. 
We follow in the footsteps
of \cite{AHLMcT}, \cite{AKMC} remarking that,
in a sense, it is to preserve
the main thrust of this Carleson argument
that we have constructed the various
technologies in this paper.
We show in this section that the argument runs
as before with some changes that are possible
as a consequence of the geometric assumptions
which we have made.

In \S5 of \cite{Morris3}, Morris illustrates
how to prove 
$$A \mapsto \int_{A} \modulus{\pri_t(x)}^2\ d\mu(x)\frac{dt}{t}$$
is a local Carleson measure.
At the heart of this proof lies
a certain test function.
Here, we illustrate how to use
the GBG condition
to  set up a suitable substitute
test function 
$\tf_{\Q,\epsilon}^w$. The main 
complication is to
choose a cutoff in a way that
we stay inside GBG coordinates
of each cube. 

Let $\tau > 1$ be chosen 
later. We note that we can find a 
smooth function $\eta_{\Q}: \cM \to [0,1]$
such that $\eta = 1$ on $B(x_{\Q}, \tau C_{1}\len(\Q))$
and $\eta = 0$ on $\cM \setminus B(x_{\Q}, 2 \tau C_{1}\len(\Q))$
and satisfying the gradient bound
$\modulus{\conn \eta} \lesssim \frac{1}{\len(\Q)}.$
The constant here depends on $\tau$.

Let $\parent{\Q}$ be the GBG cube
with centre $x_{\parent{\Q}}$.
We want to use $\eta$ to perform a
cutoff inside the ball $B(x_{\Q},\brad)$.
That is,
we want the ball 
$B(x_{\Q},2 \tau C_{1}\len(\Q)) \subset B(x_{\Q},\brad)$.
A simple calculation then yields
that we need to choose $\tau < 3$. 
Thus, for the sake of the argument, 
let us fix $\tau = 2$.

Next, let $v \in \bddlf(\C^N)$
and $\modulus{v} = 1$.
Let $\hat{w}, w \in \C^N$
such that $\modulus{\hat{w}} = \modulus{w} = 1$ 
and $\adj{v}(\hat{w}) = w$.
 Let $\tilde{w} = \eta\, w$.
Certainly, we have that
$\tilde{w} \in \Lp{2}(\cV)$
since $\tilde{w} = 0$ outside
of $B(x_{\parent{\Q}}, \brad)$.
Now, fix $\epsilon > 0$ and define
$$
\tf_{\Q,\epsilon}^w = \tilde{w} - \imath \epsilon \len(\Q) 
	\Gamma R_{\epsilon \len(\Q)}^B \tilde{w}
	= (1 + \imath \epsilon \len(\Q) \adj{\Gamma}_B R_{\epsilon \len(\Q)}^B) \tilde{w}.$$
This allows us to prove a lemma
similar to that of Lemma 5.12 in \cite{Morris3}.

\begin{lemma}
There exists $c > 0$ such that for all 
$\Q \in \DyQ_t$ with $t \leq \hscale$,
$$ \norm{\tf_{\Q,\epsilon}^w} \leq c \mu(\Q)^{\frac{1}{2}}
,\quad
\iint_{\CBox_{\Q}} \modulus{\Theta_{t}^B\tf_{\Q,\epsilon}^w}^2\ d\mu\frac{dt}{t}
	\leq c \frac{\mu(\Q)}{\epsilon^2}
\quad\text{and}\quad
\modulus{\fint_{\Q} \tf_{\Q,\epsilon}^w - w\ d\mu } 
	\leq c \epsilon^{\frac{\eta}{2}}.$$
\end{lemma}

The proof of this lemma 
proceeds exactly as the proof
of Lemma 5.12 in \cite{Morris3}.
The last estimate relies
upon Lemma \ref{Lem:Harm:Upsilon}.

Setting $\epsilon = \cbrac{\frac{1}{2c}}^{\frac{\eta}{2}}$
we obtain the same conclusion as the author
of \cite{Morris3} that
$$
\re \inprod{w, \fint_{\Q} f_{\Q}^w} \geq \frac{1}{2}$$
where we have set $f_{\Q}^w = f_{\Q,\epsilon}^w$.
Furthermore,
a stopping time argument
as in Lemma 5.11 in \cite{AKMC}
yields the following.

\begin{lemma}
Let $t_3 = \min\set{\hscale, \frac{C_{\Theta}}{4a^3 \lambda}}$.
There exists $\alpha, \beta > 0$ such that
for all $\Q \in \DyQ_t$ with $t \leq t_3$, 
there exists a collection of
subcubes $\set{\Q[k]} \subset \union_{t \leq t_3} \DyQ_t$
of $\Q$ such that 
$E_{\Q} = \Q \setminus \union_k \Q[k]$ satisfies
$\mu(E_{\Q}) \geq \beta \mu(\Q)$
and
$E^\ast_{\Q} = \CBox_{\Q} \setminus \union_k \CBox_{\Q[k]}$
satisfies
$$
\re \inprod{w, \fint_{\Q'} f^w_{\Q}} \geq \alpha
\quad\text{and}\quad
\fint_{\Q'}  \modulus{f_{\Q}^w} \leq \frac{1}{\alpha}$$
whenever $\Q' \in \union_{t \leq t_3} \DyQ_t$
with $\Q' \subset \Q$ and 
$\CBox_{\Q'} \intersect E^{\ast}_{\Q} = \emptyset$.
\end{lemma}

Let $\sigma > 0$ to be chosen later
and let $\nu \in \bddlf(\C^N)$
with $\modulus{\nu} = 1$ and define
$$K_{\nu,\sigma} = \set{ \nu' \in \bddlf(\C^N) \setminus\set{0}: 
	\modulus{ \frac{ \nu'}{\modulus{\nu'}} - \nu} \leq \sigma}.$$

\begin{proposition}
Let $t_3 = \min\set{\hscale, \frac{C_{\Theta}}{4a^3 \lambda}}$.
There exists $\sigma, \beta, c > 0$ such that
for all $\Q \in \DyQ_t$ with $t \leq t_3$, 
and $\nu \in \bddlf(\C^N)$ with $\modulus{\nu} = 1$,
there exists a collection of
subcubes $\set{\Q[k]} \subset \union_{t \leq t_3} \DyQ_t$
of $\Q$ such that 
$E_{\Q} = \Q \setminus \union_k \Q[k]$ satisfies
$\mu(E_{\Q}) \geq \beta \mu(\Q)$
and
$E^\ast_{\Q} = \CBox_{\Q} \setminus \union_k \CBox_{\Q[k]}$
satisfies
$$
\iint_{(x,t) \in E^\ast_{\Q},\ \pri_t(x) \in K_{\nu,\sigma}} 
	\modulus{\pri_t(x)}^2\ d\mu(x) \frac{dt}{t} \leq c \mu(\Q).$$
\end{proposition}

Then, the argument immediately
following Proposition 5.11 in \cite{Morris3}
illustrates that
$$ \iint_{\CBox_{\Q}} \modulus{\pri_t(x)}^2\ d\mu(x) \frac{dt}{t}
	\lesssim \mu(\Q)$$ 
and this is the required 
Carleson-measure estimate.

On choosing $t_0 = \min\set{t_2,t_3}$ and applying Proposition \ref{Prop:Harm:AvCarl}, we 
find that we have bounded all three terms on the right of \eqref{quad2}, and hence deduced the estimate \eqref{quad} of Proposition \ref{Prop:Ass:Main}: 
$$
\int_{0}^{t_0} \norm{\Theta_{t}^B P_t u}^2\  \frac{dt}{t}
	\lesssim \norm{u}^2.$$

\begin{proof}[Proof of Theorem \ref{Thm:Ass:Main}]
The invariance of (H1)-(H8) upon
replacing $(\Gamma,B_1, B_2)$
by $(\adj{\Gamma},B_2,B_1)$, $(\adj{\Gamma},\adj{B}_2,B_1)$,
and $(\adj{\Gamma}, \adj{B}_1, \adj{B}_2)$
proves Theorem \ref{Thm:Ass:Main}
via Proposition \ref{Prop:Ass:Main}.
\end{proof} 
 
By establishing Theorem \ref{Thm:Ass:Main}, 
we are now in a position to enjoy the 
full thrust of its consequences. 
The first is the Kato square
root type estimate for perturbations of 
Dirac type operators on vector bundles
as listed in Corollary \ref{Cor:Ass:Main}.
A consequence of these corollaries
is the Kato square root problem for
vector bundles, as described
in Theorem \ref{Thm:App:KatoVB}, 
the Kato square root problem for
finite rank tensors as described in Theorem
\ref{Thm:App:KatoTen}
and lastly, 
the highlighted theorem
of this paper, Theorem \ref{Thm:Int:KatoFn},
the Kato square root problem
for functions. 
Furthermore, we also can enjoy the 
holomorphic dependency results of 
\S\ref{Sect:Res:Stab}, in particular, Corollary 
\ref{Cor:Res:Stab},
which illustrates the stability of the functional 
calculus under small perturbations.

\section{Extension to Measure Metric Spaces}

In this section, we extend the
quadratic estimates to a setting where
$\cM$ is replaced by an exponentially locally doubling
measure metric space $\Spa$.
As a consequence, we also drop the
smoothness assumption on the vector bundle $\cV$.
Similar quadratic estimates
on doubling measure metric spaces
for trivial bundles are obtained by the
first author in \cite{B}.

To be precise, let $\Spa$ be a complete metric space with 
metric $d$ and let $d\mu$ be a
Borel-regular exponentially locally doubling 
measure. 
That is, we assume that
$d\mu$ satisfies (\ref{Def:Pre:Eloc})
with $\Spa$ in place of $\cM$.
The underlying space now lacks
a differentiable structure and
it no longer makes sense to 
ask the  local trivialisations 
 and the metric $\mh$ 
to be smooth.
Instead, we simply require them to be
continuous.
However, we remark that
in applications, the  local trivialisations 
 would normally be Lipschitz.
The fact that $d\mu$ is
Borel implies that the  local trivialisations are 
measurable. Furthermore, we assume that $\cV$ satisfies the GBG 
condition.

With the exception of (H6),
no changes need be made to the 
hypotheses to (H1)-(H8) in this
new setting.
To define a suitable (H6), we
first define the following 
quantity as in \cite{B}.

\begin{definition}[Pointwise Lipschitz constant]
\index{Pointwise Lipschitz constant}
For $\xi: \Spa \to \C^N$ Lipschitz, define
$\Lipp \xi: \Spa \to \R$ by
$$\Lipp \xi(x) = \limsup_{y \to x} \frac{\modulus{\xi(x) - \xi(y)}}{d(x,y)}.$$
We take the convention that
$\Lipp \xi (x) = 0$ when $x$ is an isolated point.
\end{definition}

We then define (H6) as in \cite{B}.

\begin{enumerate}
\item[(H6)]
	For every bounded Lipschitz function $\xi: \Spa \to \C$,
	multiplication by $\xi$ preserves
	$\dom(\Gamma)$ and 
	$\Mul_{\xi} = [\Gamma, \xi I]$
	is a multiplication operator.
	Furthermore, there exists a constant $m > 0$ such that
	$\modulus{\Mul_{\xi} (x)} \leq m \modulus{\Lipp{\xi}(x)}$
	for almost all $x \in \Spa$.
\end{enumerate}

Thus, we have the following theorem.
\begin{theorem}
\label{Thm:MM:Main}
Let $\Spa$ be a complete metric
space equipped with a Borel-regular
measure $d\mu$ satisfying (\ref{Def:Pre:Eloc}).
Suppose that that $(\Gamma,B_1,B_2)$
satisfy (H1)-(H8). Then,
$\Pi_B$ satisfies the quadratic
estimate
$$
\int_{0}^\infty \norm{Q_t^B u}^2\ \frac{dt}{t}
	\simeq \norm{u}^2$$ 
for all $u \in \close{\ran(\Pi_B)} \subset \Lp{2}(\cV)$
and hence has a bounded holomorphic functional 
calculus.
\end{theorem}
\begin{proof}
First, we note that
much of the local Carleson theory 
was originally proved by Morris in \S4.3 of \cite{Morris}
in the setting of an exponentially
doubling measure metric space.  
Next, the off-diagonal estimates can be obtained
by using the Lipschitz separation Lemma 5.1 
in \cite{B}. Also, the construction 
of the test function and the proof of 
Lemma \ref{Lem:Harm:Upsilon} proceeds
similar to the argument in \cite{B}.
Thus the arguments of \S\ref{Sec:Harm}
hold and the theorem is proved by 
Proposition \ref{Prop:Ass:Main}.
\end{proof}

As before, we have the following corollaries.
The $E_B^{\pm}$ are the spectral subspaces
defined in \S\ref{Sect:Res}.

\begin{corollary}[Kato square root type estimate]
\label{Cor:MM:Main}
{ \hfill }
\begin{enumerate}[(i)]
\item There is a spectral decomposition
	$$\Lp{2}(\cV) = \nul(\Pi_B) \oplus E_B^{+} \oplus E_B^{-}$$
	(where the sum is in general non-orthogonal), and
\item 	$\dom(\Gamma) \intersect \dom(\adj{\Gamma}_B) = \dom(\Pi_B) = \dom(\sqrt{{\Pi_B}^2})$
	with
	$$\norm{\Gamma u} + \norm{\adj{\Gamma}_B u} 
		\simeq \norm{\Pi_B u} 	
		\simeq \|\sqrt{{\Pi_B}^2}u\|$$
	for all $u \in \dom(\Pi_B)$.
\end{enumerate} 
\end{corollary} 

\begin{corollary}
Let $\Hil, \Gamma, B_1, B_2, \kappa_1, \kappa_2$ satisfy
(H1)-(H8) and take $\eta_i < \kappa_i$. 
Set $0 < \hat{\omega}_i < \frac{\pi}{2}$ by 
$\cos\hat{\omega}_i = \frac{ \kappa_i - \eta_i}{\norm{B_i}_\infty + \eta_i}$
and $\hat{\omega} = \frac{1}{2}(\hat{\omega}_1 + \hat{\omega}_2)$.
Let $A_i \in \Lp{\infty}(\bddlf(\cV))$ satisfy
\begin{enumerate}[(i)]
\item $\norm{A_i}_\infty \leq \eta_i$,
\item $A_1 A_2 \ran(\Gamma), B_1 A_2 \ran(\Gamma),
	A_1 B_2 \ran(\Gamma) \subset \nul(\Gamma)$, and 
\item $A_2 A_1 \ran(\adj{\Gamma}), B_2 A_1 \ran(\adj{\Gamma}),
	A_2 B_1 \ran(\adj{\Gamma}) \subset \nul(\adj{\Gamma})$.
\end{enumerate}
Letting $\hat{\omega} < \mu < \frac{\pi}{2}$,
we have:
\begin{enumerate}[(i)]
\item for all $f \in \Hol^\infty(S_{\mu}^o)$,
	$$\norm{f(\Pi_B) - f(\Pi_{B+A})}
	\lesssim (\norm{A_1}_\infty + \norm{A_2}_\infty)
	\norm{f}_\infty,\ \text{and}$$
\item for all $\psi \in \Psi(S_{\mu}^o)$,
	$$\int_{0}^\infty \norm{\psi(t\Pi_{B})u - \psi(t\Pi_{B + A})u}^2 \frac{dt}{t}
	\lesssim (\norm{A_1}_\infty^2 + \norm{A_2}_\infty^2)
	\norm{u},$$
	whenever $u \in \Hil$.
\end{enumerate}
The implicit constants depend on (H1)-(H8) and $\eta_i$. 
\end{corollary}

\bibliographystyle{amsplain}
\def\cprime{$'$}
\providecommand{\bysame}{\leavevmode\hbox to3em{\hrulefill}\thinspace}
\providecommand{\MR}{\relax\ifhmode\unskip\space\fi MR }
\providecommand{\MRhref}[2]{%
  \href{http://www.ams.org/mathscinet-getitem?mr=#1}{#2}
}
\providecommand{\href}[2]{#2}

\setlength{\parskip}{0mm}

\end{document}